\documentclass{amsart}
\usepackage{amssymb}
\usepackage{mathrsfs,mathtools}
\usepackage{graphicx}

\usepackage{lmodern} 

  \newcommand{\C}{\ensuremath{\mathbb{C}}}%
  \newcommand{\R}{\ensuremath{\mathbb{R}}}%
  \newcommand{\N}{\ensuremath{\mathbb{N}}}%
  \newcommand{\Z}{\ensuremath{\mathbb{Z}}}%
  \newcommand{\cH}{\ensuremath{\mathcal{H}}}%
  \newcommand{\tr}[1]{\ensuremath{\prescript{t}{}{#1}}}%
  \newcommand{\injtens}{\ensuremath{\otimes_\epsilon}}%
  \newcommand{\SL}{\ensuremath{\mathrm{SL}}}%
  \newcommand{\SO}{\ensuremath{\mathrm{SO}}}%
  \newcommand{\sphere}{\ensuremath{{\mathbb{S}^2}}}%
\newcommand\diag{\ensuremath{\textrm{diag}}}

\newtheorem{thm}{Theorem}[section]

\newtheorem{lemma}[thm]{Lemma}
\newtheorem{corollary}[thm]{Corollary}
\newtheorem{prop}[thm]{Proposition}

\theoremstyle{remark}
\newtheorem{rem}[thm]{Remark}

\theoremstyle{definition}
\newtheorem{dfn}[thm]{Definition}

\author{Mikael de la Salle}
\thanks{Research partially supported by the ANR projects NEUMANN and OSQPI.}

\title[Strong Banach (T) for $\SL(3,\R)$]{Towards Strong Banach property (T) for $\SL(3,\R)$}

\begin{document}
\DeclareGraphicsRule{*}{mps}{*}{} 
 \maketitle
\begin{abstract} We prove that $\SL(3,\R)$ has Strong Banach property (T) in Lafforgue's sense with respect to the Banach spaces that are $\theta>0$ interpolation spaces (for the complex interpolation method) between an arbitrary Banach space and a Banach space with sufficiently good type and cotype. 
As a consequence, every action of $\SL(3,\R)$ or its lattices by affine isometries on such a Banach space $X$ has a fixed point, and the expanders contructed from $\SL(3,\Z)$ do not admit a coarse embedding into $X$. We also prove a quantitative decay of matrix coefficients (Howe-Moore property) for representations with small exponential growth of $\SL(3,\R)$ on $X$.

This class of Banach spaces contains many superreflexive spaces and some nonreflexive spaces as well. We see no obstruction for this class to be equal to all spaces with nontrivial type.
\end{abstract}

\section{Introduction}

Kazhdan's property (T) for a topological group $G$ is a rigidity property for unitary representations of $G$. We refer to \cite{MR2415834} for more information. In recent years, U.~Bader, T.~Gelander, A.~Furman and N.~Monod \cite{MR2316269} on the one hand and V.~Lafforgue \cite{MR2423763} on the other hand independently discovered some rigidity property for actions on Banach spaces. See \cite[Section 4]{puschnigg} and \cite{nowak} for recent surveys. Our main results, Theorem \ref{thm=T_renforce_Banachique_SL3_vague} and \ref{thm=quantitative_howe_moore_vague} below, deal with Lafforgue's approach for the group $\SL(3,\R)$.

In \cite{MR2316269}, given a Banach space $X$, a property (T$_X$) was introduced in terms of (almost) invariant vectors for isometric representations on $X$, as well as a fixed point property (F$_X$) for affine isometric actions on $X$. $(\textrm{F}_X)$ implies (T$_X$), but the converse does not hold in general. The main class of spaces for which these properties were studied are $L^p$-spaces, and their subspaces/quotients (these results were generalized to noncommutative $L^p$ spaces in \cite{MR2957217}, see \cite{bekkaolivier} for a study of (T$_{\ell^p}$)). Among the results, (T$_X$) was shown to be equivalent to (T) for $X=L^p([0,1])$ whereas there are property (T) groups (for example hyperbolic groups) that do not satisfy (F$_{L^p}$) for $p$ large enough. On the opposite, higher rank algebraic groups over local fields and their lattices have property (F$_{L^p}$) for every $p \in (1,\infty)$. Bader--Furman--Gelander--Monod conjecture (\cite[Conjecture 1.6]{MR2316269}) that, for a higher rank group or its lattices, (F$_X$) holds for every superreflexive Banach space. By \cite[Proposition 8.8]{MR2316269} it is enough to prove this conjecture for higher rank groups, so it will automatically hold for their lattices.

\subsection{Strong property (T)}
In \cite{MR2423763} a stronger property, called Strong Banach property (T) was introduced. As is well-known, a group $G$ has property (T) if and only if its full $C^*$-algebra contains a projection $P$ such that, for every unitary representation $\pi$ of $G$, $\pi(P)$ is the orthogonal projection on the invariant vectors of $\pi$. In Lafforgue's Strong Banach property (T), the same definition is chosen, but for a larger class of representations~: on the one hand the representations are no longer assumed to be isometric but only of small exponential growth; on the other hand the actions do not need to be on Hilbert spaces but more generally on some Banach spaces. The word Strong refers to the first property (without it one gets just Banach property (T) as in \cite[D\'efinition 0.4]{MR2574023}), and Banach to the second. Here is a small adaptation of the original definition~: $G$ has Strong Banach (T) in the sense of \cite{MR2574023} if it has Strong Banach property (T) in the sense of the following definition, with respect to the class $\mathcal E$ of Banach spaces with nontrivial Rademacher type (see the discussion of the equivalence in subsection \ref{subsection=strongT}, and see Definition \ref{def=type_cotype} for the definitions of type and cotype). A length function on $G$ is a function $\ell:G \to \R_+$ that is continuous (or merely bounded on compacts) such that $\ell(g^{-1})= \ell(g)$ and $\ell(g_1 g_2)\leq \ell(g_1)+\ell(g_2)$ for all $g,g_1,g_2 \in G$. 

\begin{dfn}\label{def=strongT} A locally compact group $G$ has Strong Banach property (T) with respect to a class of Banach spaces $\mathcal E$ (abbreviated (T$^{Strong}_{\mathcal E}$)) if for every\footnote{it is in fact enough to consider only one length function, for example the word-length function associated to a symmetric compact generating set of $G$.} length function $\ell$ on $G$ there is a sequence of compactly supported symmetric Borel measures $m_n$ on $G$ such that, for every Banach space $X$ in $\mathcal E$, there is a constant $t>0$ such that the following holds. For every strongly continuous representation $\pi$ of $G$ on $X$ satisfying $\|\pi(g)\|_{B(X)} \leq L e^{t \ell(g)}$ for some $L\in\R_+$, $\pi(m_n)$ converges\footnote{See \eqref{eq=definition_of_pi(m)} for the definition of $\pi(m_n)$} in the norm topology of $B(X)$ to a projection on the $\pi(G)$-invariant vectors of $X$.
\end{dfn}

Strong Banach property (T) has striking consequences~: the first is that if $X$ is a Banach space and $G$ has Strong Banach property (T) with respect to $X\oplus \C$, then it has the fixed point property $(\textrm{F}_X)$ (because every affine isometric action can be realized as the restriction to the affine hyperplace $X \times \{1\}$ of a linear representation with at most linear growth, see the proof of \cite[Proposition 5.6]{MR2574023}). Also, if $G$ has (Strong) Banach property (T) with respect to $L^2(\Omega;X)$ for all measure spaces and $\Gamma$ is a residually finite lattice in $G$, then the expanders constructed from $\Gamma$ and a finite generating set of $\Gamma$ do not coarsely embed in $X$. See the proofs of \cite[Proposition 4.5, Th\'eor\`eme 5.1]{MR2423763}. Also, (Strong) Banach (T) is in general strictly stronger that the notion of \cite{MR2316269}, since a (non compact) locally compact group cannot have Strong Banach (T) (and not even Banach space (T) in the sense of \cite[D\'efinition 0.4]{MR2574023}) with respect to $L^1(G)$, whereas every (T) group has (F$_{L^1}$) by \cite{MR2929085}.

Lafforgue proved that hyperbolic groups do not satisfy Strong Banach (T) with respect to Hilbert spaces (T$^{Strong}_{Hilb}$), but for a local field $F$, $\SL(3,F)$ has (T$^{Strong}_{Hilb}$). In the case when $F$ is non-Archimedean, it was proved in \cite{MR2423763} and \cite{MR2574023} (see also \cite{liaosurvey}) that $\SL(3,F)$ satisfies Strong Banach property (T) with respect to the class of all Banach spaces of nontrivial (Rademacher) type, which is essentially the largest class of Banach spaces for which strong (T) could hold (for a non compact group). We refer to Definition \ref{def=type_cotype} for the definitions of type and cotype. Banach Strong (T) with respect to spaces of nontrivial type was recently extended to all higher rank algebraic groups over non-Archimedean local fields in \cite{liao}. This proves the conjecture in \cite{MR2316269} in the non-Archimedean case, since every superreflexive Banach space has nontrivial type by \cite{MR0394135}. See also \cite{vincentremarque} for related rigidity results.

For $\SL(3,\R)$, Lafforgue proved Strong (T) with respect to Hilbert spaces. It was clear to him and to Pisier that the same proof led to Strong Banach (T) with respect to $\theta$-Hilbertian spaces (see \S \ref{subsection=thetaHilb} for definition and details). This gives another proof of the fixed point property of $\SL(3,\R)$ or $\SL(3,\Z)$ on $L^p$ spaces \cite{MR2316269} and noncommutative $L^p$ spaces \cite{MR2957217}. However, this observation does not tell anything new regarding expanders coming from $\SL(3,\Z)$, since Pisier \cite{MR2732331} observed that a $\theta$-Hilbertian Banach space does not coarsely contain any family of expanders. The aim of the present work is to extend this result to some larger class $\mathcal E_4$ of Banach spaces. The classes $\mathcal E_r$ for $2<r<\infty$ are precisely defined in section \ref{section=main_part}. We just say here that $\mathcal E_r$ contains all spaces with sufficiently good type $p$ and cotype $q$ (namely $1/p-1/q<1/r$), and all $\theta$-Hilbertian spaces for $\theta>0$. These classes remain mysterious. We neither know whether the $\mathcal E_r$ depends on $r$, nor whether $\mathcal E_r$ contains all superreflexive spaces. For each $2<r<\infty$, $\mathcal E_r$ is made of spaces with type $>1$, and we see no obstruction for all $\mathcal E_r$ to be equal to the class of Banach spaces with type $>1$. By \cite{MR907695} each $\mathcal E_r$ contains some nonreflexive spaces.

\begin{thm}\label{thm=T_renforce_Banachique_SL3_vague} $\SL(3,\R)$ has Strong Banach property (T) with respect to $\mathcal E_4$.
\end{thm}

As we will explain at the end of this introduction, our original contribution to Theorem \ref{thm=T_renforce_Banachique_SL3_vague} is a result on the representations of $\SO(3)$ on spaces in $\mathcal E_4$ (Theorem \ref{thm=main_result_Banach}) relying on a computation (Lemma \ref{lemma=T_delta_in_Sp}) that was made in \cite{MR2838352}.


By \cite[Proposition 8.8]{MR2316269}, by the proof of \cite[Proposition 5.6]{MR2574023} and by the fact that $L^2(\Omega;X)$ belongs to $\mathcal E_4$ whenever $X$ does, this implies
\begin{corollary}
$\SL(3,\R)$ and its lattices have $(\textrm{F}_X)$ for every $X$ in $\mathcal E_4$.
\end{corollary}

By the proof of \cite[Th\'eor\`eme 5.1]{MR2423763} Theorem \ref{thm=T_renforce_Banachique_SL3_vague} implies
\begin{corollary}
Let $S$ be a finite generating set of $\SL(3,\Z)$. The expanders \[(\SL(3,\Z/n\Z),\pi_n(S))\] do not coarsely embed in Banach spaces in $\mathcal E_4$ where $\pi_n\colon \SL(3,\Z) \to \SL(3,\Z/n\Z)$.
\end{corollary}

In a work in progress with Tim de Laat we are extending the previous results to all connected higher rank simple Lie groups.

\subsection{Decay of matrix coefficients.} A locally compact group $G$ is said to have the Howe-Moore property if for every unitary representation $\pi$ of $G$, $\pi(g)$ converges in the weak operator topology to the orthogonal projection of the space of $G$-invariant vectors as $g$ tends to infinity. The Howe-Moore property does not imply (T) ($\SL(2,\R)$ has Howe-Moore), but a quantitative form of Howe-Moore does, and can lead to explicit Kazhdan constants \cite{MR1151617}, \cite{MR1905394}. The proof of Theorem \ref{thm=T_renforce_Banachique_SL3_vague} also relies on an explicit control of some of the matrix coefficients of the representations of $\SL(3,\R)$ on $X$. Pushing this method a bit further and restricting to representations without invariant vectors, one gets in Theorem \ref{thm=quantitative_howe_moore_vague} below an explicit decay of the matrix coefficients with respect to $\SO(3)$-finite vectors. This generalizes to Banach spaces and slowly growing representations the results from \cite{MR1151617}, \cite{MR1905394}. See Theorem \ref{thm=quantitative_howe_moore_precis} for an explicit form of $\varepsilon$. However, as Lafforgue pointed out to me (see Remark \ref{rem=nonoptimal}), for unitary representations these results do not lead to the optimal bound $\varepsilon$ obtained in \cite{MR1905394}. Let us also mention that Shalom proved the Howe-Moore property (without explicit decay) for isometric actions of higher rank groups on every superreflexive space (see \cite[Appendix 9]{MR2316269}).
\begin{thm}\label{thm=quantitative_howe_moore_vague}
Denote $G=\textrm{SL}(3,\R)$ and $K=\SO(3,\R)$ its maximal compact subgroup. Let $X$ be a space in $\mathcal E_4$ and $\pi$ a representation of $G$ on $X$ without invariant vectors and with small exponential growth. There is a function $\varepsilon \in C_0(G)$ such that for every $\xi\in X$, $\eta \in X^*$ of finite $K$-type\footnote{$\xi$ ($\eta$) is of finite $K$-type if the vector space generated by the $K$-orbit of $\xi$ (respectively $\eta$) is finite dimensional}
\begin{equation*}|\langle \pi(g) \xi,\eta\rangle| \leq C \varepsilon(g) \|\xi\| \|\eta\|\end{equation*}
for some constant $C$ depending only on the dimensions of the spaces $\mathrm{span}(\pi(K)\xi)$ and $\mathrm{span}(\tr{\pi}(K)\eta)$.
\end{thm}
When $X$ is a Hilbert space, the case when $\xi$ is $K$-invariant and $\mathrm{span}(\tr{\pi}(K)\eta)$ has no invariant vector was proved in \cite{MR2423763}. The case when there is no nonzero $K$-equivariant map from $\mathrm{span}(\pi(K)\xi)$ to $\mathrm{span}(\tr\pi(K)\eta)$ is an immediate adaptation of the proof. The remaining case needs some new ingredient (see the discussion before \S \ref{subsection=statements}).

\bigskip 

To end this introduction let us give a brief account of Lafforgue's method. Many approaches to rigidity for actions of $\SL(3)$ used in a way or another $\SL(2)$, through relative property (T) (\cite{MR2316269}, \cite{MR1779896}), or through asymptotics of matrix coefficients (\cite{MR1151617}, \cite{MR1905394}). The method developed in \cite{MR2423763} for $\SL(3,F)$ over a local field $F$ is different. Let us focus on the real case. The maximal compact subgroup $K=\SO(3)$ plays a key role. Let us fix once and for all a distance $d$ on $\SO(3)$ coming from a biinvariant Riemanian metric. One way to understand\footnote{My understanding of this is also influenced by \cite{HaadL} where some computations from \cite{MR2423763} are interpreted in terms of Gelfand pairs} his approach is to cut it into two main ingredients, one analytical and one combinatorical~: the analytical part of the proof is the study of the harmonic analysis of the pair of maximal compact subgroups $K=\SO(3) \subset \SL(3,\R)$ and $\SO(2) \subset \SL(2,\R)$. The combinatorial part studies the combinatorics of the various ways to see the pair $(\SO(3),\SO(2))$ inside $\SL(3,\R)$. To illustrate in a simple setting this machinery, let us outline his proof that $\SL(3,\R)$ has property (T). It already gives a fair idea of the techniques involved. Take $\pi$ a unitary representation of $\SL(3,\R)$ on $\cH$. Recall that a matrix coefficient of a unitary representation $\pi$ of a group $G$ on a Hilbert space $\cH$ is a function $g \in G\mapsto \langle \pi(g) \xi,\eta\rangle$ for $\xi,\eta\in\cH$. We say the matrix coefficient is normalized if $\xi,\eta$ are unit vectors. For a subgroup $H$ of $G$ we say that a function $f$ on $G$ is $H$-biinvariant if $f(hgh')= f(g)$ for all $g \in G$ and $h,h' \in H$. The analytical part is the following fact~: every $U=\begin{pmatrix} 1 & 0\\ 0 & \SO(2)\end{pmatrix}$-biinvariant normalized coefficient of a unitary representation of $K$ is H\"older continuous with exponent $1/2$ on a neighbourhood of $x_0 = \begin{pmatrix} 0 & -1 & 0 \\ 1&0&0\\0&0&1\end{pmatrix}$ (and in fact on a neighbourhood of every element of $K$ except those that are block diagonal in the decomposition $\R^3 \simeq \R\oplus \R^2$), with uniform constant (see \eqref{eq=U-invariant_coeff_hoelder} for details). Now if $D \in \SL(3,\R)$ commutes with every element of $U$, and if $\xi,\eta \in \cH$ are $K$-invariant unit vectors, $x \in K \mapsto \langle \pi(DxD)\xi,\eta\rangle$ is a $U$-biinvariant normalized matrix coefficient of $\pi$, and hence for all $x \in K$
\begin{equation}\label{eq=in_the_proof_of_T} |\langle \pi(Dx_0D)\xi,\eta\rangle - \langle \pi(DxD)\xi,\eta\rangle| \leq C d(x,x_0)^{1/2}.\end{equation}
If we agree to say that the price to pay for jumping from $D x_0 D$ to $D x D$ (or from $D x D$ to $D x_0 D$) is $C d(x,x_0)^{1/2}$, the combinatorial part of the proof reads as follows~: Lafforgue is able to explore the whole group $\SL(3,\R)$ (or rather the space of cosets $K \backslash \SL(3,\R)/K$), in such a way that any pair $g,g'\in \SL(3,\R)$ \emph{far from the identity} can be linked together by a series of jumps of the form just described, with total cost \emph{very small} (see subsection \ref{subsection=first_step} for details). Together with \eqref{eq=in_the_proof_of_T}, this implies that $g \mapsto \langle \pi(g)\xi,\eta\rangle$ satisfies the Cauchy criterion (uniformly in $\xi,\eta$), and hence that it converges as $g$ escapes to infinity in $G$. Hence, $\int_{K\times K} \pi(kgk) dk$ has a limit, say $P$, as $g \to \infty$ in the norm topology of $B(\cH)$. The last (easy) step is to identify $P$ as the orthogonal projection on the space of $G$-invariant vectors (see \eqref{eq=pi(KgP)=P} and the following discussion for details). In particular if the representation $\pi$ we started with had no invariant vectors, $P=0$, and $\| \int_{K\times K} \pi(kgk) dk\|_{B(\cH)} \leq 1/2$ for some $g \in G$. This clearly implies $\pi$ does not almost have invariant vectors, and proves that $\SL(3,\R)$ has property (T). The proof of Strong property (T) uses similar ideas, but is significantly more involved since one also has to deal with non $K$-biinvariant coefficients to replace the last step that is no longer easy at all.

The fact that all the analysis is done on the level of compact groups is what allows to consider also representations with small growth, but it also has a remarkable consequence in harmonic analysis/operator algebra (\cite{MR2838352}, \cite{HaadL}, \cite{dLaat}). Indeed, since for a compact group $K$, the space of completely bounded multipliers of the Fourier algebra\footnote{for a locally compact group $G$, $A(G)$ is the space of matrix coefficients of the representation $\lambda$ of $G$ on $L^2(G)$ by left translation, with norm defined by $\|\psi\|_{A(G)} \leq 1$ if and only if there are unit vectors $\xi,\eta \in L^2(G)$ such that $\psi(x) = \langle \lambda(x) \xi,\eta\rangle$. A completely bounded multiplier of $A(G)$ is a function $f$ on $G$ such that $(g,k) \mapsto f(g)\psi(g,k) \in A(G\times K)$ for all $\psi \in A(G\times K)$ and all compact groups $K$.} $A(K)$ coincides with $A(K)$, the H\"older continuity of $U$-biinvariant matrix coefficients of $K$ can be rephrased as H\"older continuity of $U$-biinvariant completely bounded multipliers of the Fourier algebra $A(K)$, and the proof sketched above gives that not only the $C_0$ $K$-biinvariant matrix coefficients of unitary representations of $\SL(3,\R)$, but also the $C_0$ $K$-biinvariant completely bounded multipliers of $A(\SL(3,\R))$ have an explicit decay at infinity. Haagerup and de Laat \cite{HaadL} were able to deduce from this, with a very short proof, that $\SL(3,\R)$ does not have the Approximation Property. This was already known as a consequence of \cite{MR2838352}, where we also used Lafforgue's machinery to prove that the non-commutative $L^p$ spaces of the von Neumann algebras of lattices in $\SL(3,\R)$ fail the completely bounded approximation property for all $p \in (4,\infty]$. If one looks at the proofs, one remarks that the $4$ here is the same $4$ as in the definition of the class $\mathcal E_4$.

If one tries to adapt these techniques to representations of $\SL(3,\R)$ on a Banach space $X$, the difficulty lies in the analytical part. Indeed, by its combinatorial nature, the combinatorial part of the proof extends without any problem to Banach spaces. The issue is to prove H\"older continuity of $U$-biinvariant matrix coefficients of arbitrary $X$-valued representations of $K$. An easy observation (Proposition \ref{prop=all_reduces_to_regular}) allows to reduce this task to just the left regular representation on the Bochner space $L^2(K;X)$ of $L^2$ functions from $K$ to $X$, with $K$ acting by translation. All this was well known to Lafforgue (see question b in the introduction of \cite{MR2574023}). Our main result is therefore the following (see \eqref{eq=def_of_Tdelta} for the definition of $T_\delta$)~: 

\begin{thm}\label{thm=main_result_Banach}
If $X$ be a Banach space in the class $\mathcal E_{4}$, then there is a constant $C>0$ and $s>0$ such that for every $\delta \in [-1,1]$
\[ \|(T_0 - T_\delta)\otimes \mathrm{Id}_X\|_{B(L^2(K;X))} \leq C |\delta|^{s}.\]
$2s$ is the product of all the $\theta$'s that appear in the process that produces $X$ from a space with type $p$ and cotype $q$ satisfying $1/p-1/q<1/4$.
\end{thm}

The fact that Theorem \ref{thm=main_result_Banach} implies Theorem \ref{thm=T_renforce_Banachique_SL3_vague} is due to Lafforgue. The proof has not been written, but it is an adaptation of the proof in the case of Hilbert spaces in \cite[Th\'eor\`eme 2.1]{MR2423763}. For the reader's convenience we give a detailed proof.

\subsection{Organization of the paper} In \S \ref{section=notation} we set the notation, collect preliminary facts on Banach spaces, representations of topological groups on Banach spaces and in particular compact groups, and on the representation theory of the pair $(\SO(3),\SO(2))$. \S \ref{section=main_part} contains the proof of Theorem \ref{thm=main_result_Banach}. \S \ref{section=Lafforgue_theorem} contains a detailed description of Lafforgue's method with some new variants. It contains the proofs of Theorem \ref{thm=T_renforce_Banachique_SL3_vague} and \ref{thm=quantitative_howe_moore_vague}.

\paragraph{Acknowledgements} The questions I address in this paper were successively asked to me by Gilles Pisier, Vincent Lafforgue and Uri Bader. This paper owes a lot to Gilles Pisier. I thank warmly Uri Bader for inviting me to Israel and for extremely stimulating discussions. I thank him, Tim de Laat and the anonymous referees for very useful comments. I also thank Christian Le Merdy and \'Eric Ricard for useful discussions at a preliminary stage of this work.

\section{Notation and Preliminaries}\label{section=notation}

All the Banach spaces will be over $\C$.

For numerical expressions $A, B$ (allowed to depend on some parameters), we will write $A \lesssim B$ to denote $A \leq C B$ for some universal constant
$C>0$. We will write $A \lesssim_i B$ to denote $A \leq C_i B$ for $C_i>0$ depending only on the parameter $i$.

We could not find a reference for all the results in this section (in particularfor the first half of Theorem \ref{thm=peterWeyl}), but we believe that none of them are new, except Proposition \ref{prop=nonspherical_Holder_general}.

\subsection{Banach space valued $L^p$ spaces} When $X$ is a Banach space and $(\Omega,\mu)$ a measure space, a function $f:\Omega \to X$ is said to be Bochner measurable if it is the almost everywhere limit of a sequence of measurable finite valued functions from $\Omega$ to $X$.

$L^p(\Omega,\mu;X)$ or $L^p(\Omega;X)$ denotes the Bochner space of
$X$-valued $L^p$ functions, \emph{i.e.} the space of Bochner measurable functions such that \[\|f\|_p:=\left(\int \|f(\omega)\|_X^p d\mu(\omega)\right)^{1/p}<\infty.\] $L^p(\Omega,\mu;X)$ can also be seen as the completion of the space of
simple functions $t \in \Omega \mapsto \sum \chi_{E_i}(t)x_i$ (finite
sum with $x_i \in X$ and $E_i \subset \Omega$ measurable with finite measure) with respect
to the seminorm $\left(\int_\Omega \|f(t)\|_X^p d\mu(t)
\right)^{1/p}$.

In the particular case when $p=1$, the integral of $f \in L^1(\Omega,\mu;X)$, written $\int f d\mu \in X$, is well defined on the dense subspace of functions with finite values, and extends by density to a norm $1$ map. For more background on Banach space valued $L^p$ spaces, see \cite{MR0453964}.

When $T:L^p(\Omega,\mu) \to L^p(\Omega,\mu)$ we will denote by $T_X$
the operator $T \otimes id_X$ (a priori only defined on
the dense subspace $L^p \otimes X$ of $L^p(\Omega;X)$). If $T_X$
extends to a bounded operator on $L^p(\Omega;X)$ we will denote by 
$\|T_X\| = \|T \otimes id_X\|_{B(L^p(\Omega;X))}$ its norm. Otherwise
we set $\|T_X\|=\infty$. When $X$ is an $L^p$-space, $\|T_X\|=\|T\|$ by Fubini's theorem. In particular, if $p=2$ and $X$ is a Hilbert space, $\|T_X\| = \|T\|$. The operators for which $\|T_X\|<\infty$ for every Banach space $X$ have several equivalent characterizations, that we recall.

\subsection{Regular operators}\label{subsection=regular} Let $(\Omega,\mu)$ be a measure space and $1<p<\infty$. 

A linear operator $T:L^p(\Omega,\mu) \to L^p(\Omega,\mu)$ is called
\emph{regular} if one of the following equivalent conditions hold~:
\begin{itemize}
\item There is a constant $C$ such that for all Banach spaces $X$,
  $\|T_X\| \leq C$.
\item There is a constant $C$ such that $\|T_{\ell^\infty}\| \leq C$.
\item There is a bounded operator $S:L^p(\Omega,\mu)
  \to L^p(\Omega,\mu)$ such that $|T(f)| \leq S(|f|)$ for all $f$.
\item Up to a change of the measure $\mu$, $T$ is simultaneously
  bounded on $L^1(\Omega,\mu)$ and $L^\infty(\Omega,\mu)$.
\end{itemize}

The best $C$ is then equal to $\inf_S \|S\|_{L^p \to L^p}$ and to
the infimum over the changes of measures of $\max (\|T\|_{L^1 \to L^1}, \|T\|_{L^\infty \to L^\infty})$, and will be denoted by
$\|T\|_{B_r(L^p)}$. By change of measure, we mean that we replace
$\mu$ by an equivalent measure $\nu$ and $T$ by $ u \circ T \circ
u^{-1}$ where $u:L^p(\Omega,\mu) \to L^p(\Omega, \nu)$ is the isometry
given by $uf(t) = (d\mu/d\nu)(t)^{1/p} f(t)$. See \cite{MR2732331} and the references therein for details.

\subsection{Schatten classes} 
For a Hilbert space $H$, and $1\leq p < \infty$, the Schatten
$p$-class $S^p(H)$ (or $S^p$) is the space of operators on $H$ such
that $ (T^* T)^{p/2}$ is of trace class. It is a Banach space for the
norm $\|T\|_{S^p} = Tr ((T^* T)^{p/2})^{1/p}$. $S^\infty$ is the space
of compact operators with operator norm. Note that if $p<q$ we have
a contractive inclusion $S^p \subset S^q$~: it is harder to be in
$S^p$ than in $S^q$.

\subsection{Interpolation}\label{subsection=interpolation} 
Everything that will be used from the theory of interpolation of
Banach spaces can be found in \cite{MR0482275}. For convenience we briefly recall the definition. A compatible couple $(X_0,X_1)$ is a pair of Banach spaces together with continuous linear embeddings of $X_0$ and $X_1$ in a same topological vector space. Through this embedding, the sum $x_0+x_1$ of an element $x_0$ of $X_0$ and $x_1$ of $X_1$ makes sense as an element of the underlying topological vector space, and we define the norm of such an element as the infimum over all such decompositions of $\|x_0\|_{X_0} + \|x_1\|_{X_1}$. We get a Banach space denoted $X_0+X_1$. Consider the Banach space of functions $f:\overline S=\{z \in \C, \Re(z) \in [0,1]\} \to X_0+X_1$ that are bounded continuous on $\overline S$, holomorphic in the interior of $\overline S$, and such that the restrictions $t \in \R \mapsto f(k+it)$ belong to $C_0(\R;X_k)$ for $k=0$ and $k=1$. The norm of $f$ is $ \sup_{t \in \R,k \in \{0,1\}}\|f(k+it)\|_{X_k}$. Let $\theta \in (0,1)$. By Hadamard's three line lemma the map $f \mapsto f(\theta) \in X_0+X_1$ is continuous. The complex interpolation space $[X_0,X_1]_{\theta}$ is defined as the image of this map, with norm the quotient norm. For example, if we fix a measure space, the classical Riesz-Thorin theorem can be expressed in this setting as the isometric equalities $L^{p_\theta} = [L^{p_0},L^{p_1}]_\theta$ for every $p_0,p_1,p_\theta \in [1,\infty]$ satisfying $1/p_\theta = (1-\theta)/p_0+\theta/p_1$. Analogous results hold for noncommutative $L^p$ spaces. The fundamental property of complex interpolation is the following~: if $(X_0,X_1)$ and $(Y_0,Y_1)$ are compatible couples, and if a map $T:X_0 + X_1 \to Y_0 + Y_1$ maps $X_0$ in $Y_0$ and $X_1$ in $Y_1$, then it also maps $[X_0,X_1]_\theta$ in $[Y_0,Y_1]_\theta$, with norm at most $\|T\|_{X_0\to Y_0}^{1-\theta} \|T\|_{X_1\to Y_1}^{\theta}$. One says that the functor $\mathcal C_\theta$ from the category of compatible pairs of Banach spaces to the category of Banach spaces that associates $[X_0,X_1]_{\theta}$ to $(X_0,X_1)$ is an \emph{exact interpolation functor of exponent $\theta$}. The additional property of this functor that we use is that it \emph{commutes with vector-valued $L^p$-spaces}~: for every compatible pair $(X_0,X_1)$ and every $p_0,p_1 \in [1,\infty]$, $[L^{p_0}(\Omega;X_0),L^{p_1}(\Omega;X_1)]_\theta = L^{p_\theta}(\Omega;[X_0,X_1]_\theta)$ with $1/p_\theta = (1-\theta)/p_0 + \theta/p_1$ through the natural identifications. 
\begin{rem}\label{rem=real_interpolation}
We choose to talk about complex interpolation, but the only properties of the complex interpolation functor $\mathcal C_\theta$ that we use are the following
\begin{enumerate}
\item \label{item=exponenttheta} It is a functor of exponent $\theta>0$.
\item \label{item=commute_avec_L2} For any measure space and every compatible pair $(X_0,X_1)$, 
\[L^2(\Omega;\mathcal C_\theta(X_0,X_1)) = \mathcal C_\theta(L^2(\Omega;X_0),L^2(\Omega;X_1)) \textrm{   (equivalent norm)}.\]
\end{enumerate}
In fact by a small adaptation of Lemma \ref{lemma=stability_under_subquotients_ultrapr} we could as well work with any interpolation functor $\mathcal C_\theta$ satisfying (\ref{item=exponenttheta}) and (\ref{item=commute_avec_L2}) replaced by $\exists p \in (1,\infty)$ such that, $L^p(\Omega;\mathcal C_\theta(X_0,X_1)) = \mathcal C_\theta(L^p(\Omega;X_0),L^p(\Omega;X_1))$ (equivalent norm). By \cite[5.8.2]{MR0482275} these properties also hold for the real interpolation functor $(X_0,X_1)\mapsto [X_0,X_1]_{\theta,p}$ (for $1<p<\infty$). 
\end{rem}
\subsection{Superreflexive spaces} A Banach space $X$ is called superreflexive if all its ultrapowers are reflexive. Equivalently, if $X$ is isomorphic to a uniformly convex Banach space. A Banach space is uniformly convex if
\[ \sup\left\{ \|u+v\|/2, \|u\|,\|v\|\leq 1, \|u-v\| \geq  \varepsilon\right\} <1 \textrm{ for all }\varepsilon>0.\]

\subsection{Rademacher Type and cotype} We briefely recall definitions and properties of type and cotype. For details on this, proofs, references and more, see \cite{MR1999197}.

Let $(g_i)_{i \in \N}$ be a sequence of independent complex gaussian
$\mathcal N(0,1)$ random variables defined on some probability space
$(\Omega,\mathbb P)$.

\begin{dfn}\label{def=type_cotype}
A Banach space $X$ is said to have type $p\geq 1$ if there is a constant $T$
such that for all $n$ and all $x_1,\dots,x_n \in X$,
\[ \| \sum_i g_i x_i\|_{L^2(\Omega;X)} \leq T \left(\sum_i \|x_i\|^p\right)^{1/p}.\]
The best $T$ is denoted by $T_p(X)$. 

A Banach space $X$ is said to have cotype $q\leq\infty$ if there is a constant
$C$ such that for all $n$ and all $x_1,\dots,x_n \in X$,
\[ \left(\sum_i \|x_i\|^q\right)^{1/q} \leq C \| \sum_i g_i x_i\|_{L^2(\Omega;X)} .\]
The best $C$ is denoted by $C_q(X)$. 
\end{dfn}
Note that if $X \neq \{0\}$ has type $p$ and cotype $q$, necessarily $p \leq 2$
and $q \geq 2$ (take $x_i$ all equal). By H\"older's inequality, it also has type $\widetilde p$ for every $1\leq \widetilde p \leq p$ and cotype $\widetilde q$ for every $\widetilde q \geq q$. Every Banach space has type $1$ and cotype $\infty$ with constant $1$. A Banach space is therefore said to have nontrivial type (cotype) if it has type $p>1$ (respectively cotype $q<\infty$). Hilbert spaces have type $2$ and cotype $2$, and by a theorem of Kwapie{\'n} \cite{MR0341039} this property characterizes the Banach space that are isomorphic to Hilbert spaces. Superreflexive spaces have nontrivial type, but the converse is not true~: there are spaces of nontrivial type that are not even reflexive. More importantly for our purposes, for every $q>2$ there are Banach spaces that are not reflexive but have type $2$ and cotype $q$ \cite{MR907695}.

By the properties of complex interpolation recalled above, type is stable under complex interpolation, and more precisely if $(X_0,X_1)$ is a compatible pair with $X_k$ of type $p_k$, then $[X_0,X_1]_\theta$ has type $p_\theta$ with $1/p_\theta = (1-\theta)/p_0 + \theta/p_1$. In particular if one of $X_0$ or $X_1$ has nontrivial type, every intermediate space will also have nontrivial type. Moreover type and cotype behave nicely with respect to duality~: if a Banach space has type $p>1$ and cotype $q$, its dual has type $q'$ and cotype $p'$ for the conjugate exponents of $q$ and $p$ respectively. 

In the definitions we can replace the independent gaussian variables
by independent Rademacher (uniformly distributed in $\{-1,1\}$)
variables, and this explains the terminology. We can also replace the norm $L^2$ by $L^r$ for any other $r \in (1,\infty)$. We get the same notion, but with different values of
$T_p(X)$ and $C_q(X)$.

\subsection{Representations} 
A linear continuous representation of a locally compact group $G$ equipped with a left Haar measure on a Banach space $X$ is a group morphism from $G$ to the invertible operators in $B(X)$ that is continuous if $B(X)$ is equipped with the strong operator topology. Unless explicitly specified, we will always assume that the linear representations that we are considering are continuous. A Banach $G$-space is a Banach space with a linear continuous representation of $G$.

When $m$ is a compactly supported signed Borel measure on $G$ and $\pi$ is a continuous representation of $G$ on a Banach space $X$, we will denote $\pi(m) \in B(X)$ the operator defined by
\begin{equation}\label{eq=definition_of_pi(m)} \pi(m) \xi = \int \pi(g)\xi dm(g) \textrm{ (Bochner integral) }\forall \xi \in X.\end{equation}
This definition makes sense because by the definition of the strong operator topology, $g\mapsto \pi(g) \xi$ is continuous. We might sometimes abusively write $\pi(m) = \int \pi(g) dm(g)$ when we mean \eqref{eq=definition_of_pi(m)}. When $f \in C_c(G)$ we will denote by $\pi(f)$ the value of \eqref{eq=definition_of_pi(m)} for the measure $f dg$.

By abuse, if $G$ is compact and $m$ is the Haar probability measure on $G$, we will denote $\pi(m) = \int_G \pi(g) dm(g)$ by $\pi(G)$. Similarly, if $K$ is a compact subgroup of $G$, we will write $\pi(KgK) = \pi(K) \pi(g) \pi(K)$  ($\pi(Kg)=\pi(K)\pi(g)$, $\pi(gK) = \pi(g) \pi(K)$) the operator corresponding to the $K$-biinvariant (left, right $K$-invariant) probability measure on the coset $KgK$ (respectively $Kg$, $gK$). 
 
The contragredient representation $\tr \pi$ of a representation $\pi$ of $G$ on $X$ is usually defined as the representation of $G$ on $X^*$ given by $g \in G \mapsto \pi(g^{-1})^*$. Even if $\pi$ is continuous, this representation is not necessarily continuous if $G$ is not discrete and $X$ is not reflexive (think of $G$ acting by translation on $L^1(G)$), we therefore prefer here to define $\tr \pi$ as the restriction of the representation $g \in G \mapsto \pi(g^{-1})^*$ to the set of vectors $x \in X^*$ such that $g \mapsto \pi(g^{-1})^*x$ is continuous. This makes sense by the
\begin{lemma}\label{lemma=contragredient} If $\pi$ is a continuous representation of $G$ on a Banach space $X$, the set of $x \in X^*$ such that $g \mapsto \pi(g^{-1})^*x$ is continuous is a weak-* dense Banach subspace invariant by $\pi(g^{-1})^*$ for every $g \in G$.
\end{lemma}
\begin{proof} It is clear that the set of such vectors forms a closed vector space invariant by $\pi(g^{-1})^*$ for every $g \in G$. To prove that it is weak-* dense we first prove that it contains all vectors of the form $\pi(f)^*x$ for $f \in C_c(G)$ and $x \in X$. Indeed, $g \mapsto f \ast \delta_{g^{-1}}$ is continuous as a map from $G$ to the compactly supported measures on $G$, so that $g \mapsto \pi( f) \pi(g)$ is continuous as a map from $G$ to $B(X)$ equipped with the norm topology, which clearly implies that $g \mapsto \pi(g^{-1})^* \pi(f)^* x$ is continuous. It remains to prove that $\{ \pi(f)^* x, x \in X, f \in C_c(G)\}$ is weak-* dense. Let $y\in X$ such that $\langle y, \pi(f)^* x \rangle =0$ for all such $f$ and $x$. Taking $f_n \geq $ with support converging to $\{1\}$ and $ \int f_n =1$, we get $\langle y,x\rangle =\lim_n \langle \pi(f_n) y, x\rangle = 0$ for all $x \in X$. Hence $y=0$, and this proves the claimed density.
\end{proof}

\subsection{Representations of compact group on Banach spaces}
Let $K$ be a compact group. By the Peter-Weyl theorem, every unitary representation of $K$ decomposes as a direct sum of irreducible representations. For isometric (or arbitrary continuous) representations on Banach spaces, a similar phenomenon holds but in a weaker form. We recall this here.

We will use the following. One form of the Peter-Weyl theorem provides an orthonormal decomposition $L^2(K)=\oplus_{V} E_V$ where the sum is indexed by the equivalence classes of irreducible unitary representations of $K$ and $E_V$ is the space of coefficients of $V$. The irreducible unitary representation of $K$ on $V$ is denoted by $\pi_V$, and the space of coefficients of $V$ is the linear space spanned by the coefficients $k \mapsto \langle \pi_V(k) \xi ,\eta\rangle$ for $\xi, \eta \in V$. The orthogonal projection from $L^2(K)$ on the space of coefficients of $V$ is given by the left or right convolution by $d_V \chi_V$, where $d_V$ is the dimension of $V$ and $\chi_V$ is the character of $V$, that is the function $\chi_V\colon k \mapsto Tr(\pi_V(k))$.

\begin{dfn}\label{dfn=finite_K_type}
Let $\pi$ be a continuous representation of a compact group $K$ on a Banach space $X$. 

A vector $\xi \in X$ is of finite $K$-type if $\mathrm{span}(\pi(K) \xi)$ is finite dimensional.

If $V$ is an irreducible unitary representation of $K$, a vector $\xi \in X$ is called of $K$-type $V$ if for every $\eta \in X^*$, the coefficient $k \mapsto \langle \pi(k) \xi,\eta\rangle$ belongs to the space of coefficients of $V$. We denote $X_V$ the vector space consisting of vectors of $K$-type $V$.
\end{dfn}
By an easy application of Hahn-Banach, a vector of $K$-type $V$ is of finite $K$-type. Indeed $\mathrm{span}(\pi(K) \xi)$ has dimension of at most the dimension of the space of coefficients of $V$, that is $d_V^2$.

We will need the following generalization of the usual Peter-Weyl theorem to actions of compact groups on Banach spaces. In this statement $X_1 \injtens X_2$ denotes the injective tensor product of the Banach spaces $X_1$ and $X_2$, one of the natural Banach space completions of $X_1 \otimes X_2$. We recall that when $X_2$ is finite dimensional $X_1 \injtens X_2$ is naturally isometric with the Banach space of bounded linear maps from $X_2^*$ to $X_1$. This precise choice of a norm on the tensor product has no importance except for the precise value of the constants.
\begin{thm}[Peter-Weyl Theorem]\label{thm=peterWeyl} Let $\pi$ be a continuous isometric representation of a compact group $K$ on a Banach space $X$.
\begin{itemize}
\item For every irreducible unitary representation $V$ of $K$, $\pi(d_V \overline{\chi_V})$ is a $K$-equivariant projection on $X_V$. 
\item There is a Banach space $Y$ and a $K$-equivariant isomorphism $u:X_V \to Y \injtens V$ ($K$ acting trivially on $Y$) satisfying $\|u\| \leq d_V$, $\|u^{-1}\|\leq d_V$.
\item The space $X_{finite}$ of vectors of finite $K$-type is the direct sum of the subspaces $X_V$, for $V$ in the equivalence classes of irreducible unitary representations of $K$.
\item $X_{finite}$ is dense in $X$.
\end{itemize}
\end{thm}
The third and fourth items were proved by Shiga \cite{MR0082624}.
\begin{proof}
Let $V$ an irreducible representation of $K$, and let us prove that $\pi(d_V \overline{\chi_V})$ is a projection on $X_V$. Denote by $P_V$ the orthogonal projection from $L^2(K)$ onto the space of coefficients of $V$, and recall that $P_V$ coincides with the right (and also left) convolution by $d_V \chi_V$. Let $\xi \in X$, $\eta \in X^*$. Then for every $k_0 \in K$, using that $\overline{\chi_V(k)} = \chi_V(k^{-1})$,
\begin{align*} \langle \pi(k_0) \pi(d_V \overline{\chi_V}) \xi,\eta\rangle &= \int_K \langle \pi(k_0 k)\xi,\eta\rangle d_V \chi_V(k^{-1}) dk\\ & =  P_V(k \mapsto \langle \pi(k) \xi,\eta\rangle)(k_0). \end{align*}
This shows that $\pi(d_V \overline{\chi_V}) \xi$ belongs to $X_V$, and that $\pi(d_V \overline{\chi_V}) \xi=\xi$ if $\xi \in X_V$. $\pi(d_V \overline{\chi_V})$ is therefore a projection on $X_V$. It is clearly equivariant.

Define $Y$ as the Banach space of $K$-equivariant bounded maps from $V$ to $X$. The map $y \otimes v\in Y\otimes V \mapsto (T \mapsto y(T v)) \in B(S^1(V),X)$ induces an isometric isomorphism from $Y \injtens V$ onto the space $E$ of linear maps $S^1(V) \to X$ that are $K$-equivariant for the left action of $K$ on the Schatten class $S^1(V)$. Note for later use that an element $u \in E$ is in fact a map from $S^1(V)$ to $X_V$, because for all $T \in S^1(V)$ and $\eta \in X^*$, $\langle \pi(k) u(T),\eta\rangle = \langle \pi_V(k) T, u^* \eta\rangle$ is a coefficient of $V$. Through this isomorphism, the action of $K$ on $Y \otimes V$ (trivial action on $Y$, natural action on $V$) corresponds to the action on $E$ given by $k \cdot u(T) = u(T \pi_V(k))$ for every $u \in E$ and $T \in S^1(V)$. We now construct an isomorphism of $K$-Banach spaces between $E$ and $X_V$, this will conclude the proof of the second item. Consider $\Phi: u \in E \mapsto u(\mathrm{Id}) \in X$. For $\xi \in X$, define $\Psi(\xi) \in E$ by
\[\Psi(\xi)(T) = d_V \int_K Tr( \pi_V(k)^*T) \pi(k) \xi\ \  \forall T \in S^1(V).\]
This defines $K$-equivariant maps $\Psi:X \to E$ and $\Phi:E \to X$ satisfying $\|\Phi\|\leq d_V$ and $\|\Psi\|\leq d_V$. The formula $\Phi \circ \Psi(\xi) = \int_K d_V Tr(\pi_V(k)^*) \pi(k) \xi$ means that $\Phi \circ \Psi = \pi(d_V \overline{\chi_V})$ is the $K$-equivariant projection on $X_V$. It remains to check that $\Psi \circ \Phi$ is the identity of $E$. We have to show that $d_V \int_K Tr(\pi_V(k)^* T) \pi(k)u(\mathrm{Id}) = u(T)$ for all $T \in B(V)$. Since $V$ is irreducible, $\pi_V(K)$ spans $B(V)$, and hence it suffices to prove this equality with $T=\pi_V(k_0)$ for $k_0 \in K$. Then we have
\[d_V \int_K Tr(\pi_V(k)^* \pi_V(k_0)) \pi(k)u(\mathrm{Id}) = d_V \pi(\overline{\chi_V})\pi(k_0)u(\mathrm{Id})= u(\pi_V(k_0)).\]
This first equality was the change of variable $k k_0^{-1}$, the second equality is because $\pi(k_0) u(\mathrm{Id}) = u(\pi_V(k_0))$ belongs to $X_V$ and $\pi(d_V \overline{\chi_V})$ is a projection on $X_V$.

We now prove that $X_{finite}$ is the direct sum of the $X_V$'s. By the orthogonality relation of characters of $K$, $\pi(d_V \overline{\chi_V}) \pi(d_W \overline{\chi_W})=0$ if $V$ and $W$ are non-equivalent irreducible representations of $K$. This shows that the spaces $X_V$ are in direct sum. As already noted, they are contained in $X_{finite}$. It remains to prove that every element $\xi$ of $X_{finite}$ belongs to $\oplus_V X_V$. If the finite dimensional representation $\mathrm{span}(\pi(K) \xi)$ is irreducible (say it is isomorphic to $V$), then $\xi$ belongs to $X_V$. In general the finite dimensional representation $\mathrm{span}(\pi(K) \xi$ decomposes as a finite direct sum of irreducible representations. In particular $\xi$ decomposes as a finite sum of vectors that each belong to one of the $X_V$'s.

Let us show that $X_{finite}$ is dense. Take $\eta \in X^*$ that vanishes on $X_{finite}$, and take $\xi \in X$ arbitrary. The function $f:k \in K \mapsto \langle \pi(k) \xi,\eta\rangle$ satisfies $P_V f(k) = \langle d_V \overline{\chi_V}(\pi) \pi(k) \xi,\eta\rangle = 0$ for every irreducible representation of $G$. By the Peter-Weyl theorem this function is zero in $L^2(K)$ and hence identically zero because it is continuous. In particular $\langle \xi,\eta\rangle=0$. Since $\xi$ is arbitrary, this shows that $\eta=0$, and hence $X_{finite}$ is dense by the Hahn-Banach theorem.
\end{proof}
\begin{rem}\label{rem=finite_type_on_dual} On the dual of $X$, the space $X^*_V$ can be defined in two (or three) equivalent ways. Either as the space of vectors $\xi \in X^*$ such that $k \mapsto \langle \pi(k) \xi,\eta\rangle$ belongs to the space of coefficients of $V$ for every $\eta \in X$ (or every $\eta \in X^{**}$), or as the space of vectors in the space $Y \subset X^*$ of the contragradrient representation of $\pi$ such that $k \mapsto \langle \pi(k) \xi,\eta\rangle$ for all $\eta$ in $Y^*$. Then by Theorem \ref{thm=peterWeyl} and Lemma \ref{lemma=contragredient}, the space $X^*_{finite} = \oplus_V X^*_V$ is a weak-* dense subspace of $X^*$.
\end{rem}

The following is a generalization to Banach spaces of the fact that every unitary representation of a compact group is weakly contained in the regular representation, and will be used in Section \ref{section=Lafforgue_theorem}.
\begin{prop}\label{prop=all_reduces_to_regular} Let $\pi$ be an isometric (strongly continuous) representation of a compact group $K$ on a Banach space $X$. Then for every signed Borel measure $m$ on $K$,
\[ \| \pi(m)\|_{B(X)} \leq \| \lambda(m)_X\|_{B(L^2(K;X))}.\]
\end{prop}
\begin{proof}
Let us identify $X$ isometrically as the subspace of constant functions in $L^2(K;X)$. This realizes $\pi$ as a subrepresentation of $\lambda \otimes \pi$ on $L^2(K;X)$, and hence 
\[\| \pi(m)\|_{B(X)} \leq \| (\lambda \otimes \pi)(m)\|_{B(L^2(K;X))}.\]
The surjective isometry $V$ of $L^2(K;X)$ defined by $Vf(k)=\pi(k)f(k)$ satisfies $\lambda \otimes \pi = V \circ (\lambda \otimes \mathrm{Id}) \circ V^{-1}$, which gives
\[\| (\lambda \otimes \pi)(m)\|_{B(L^2(K;X))} = \| \lambda(m)_X\|_{B(L^2(K;X))}.\qedhere\]
\end{proof}

In the next result we consider a closed subgroup $U$ of $K \times K$ acting on $K$ by left-right multiplication, that is by $(k,k')\cdot x = kxk'^{-1}$. This result will be useful in the situation when the $U$-invariant matrix coefficients of $K$ are known to have some nice continuity behaviour, and will allow to deduce similar behaviour but for more general $U$-equivariant coefficients. This will play a key role in section \ref{section=Lafforgue_theorem}, for $K=\SO(3,\R)$ and $U = \SO(2,\R) \times \SO(2,\R)$.

\begin{prop}\label{prop=nonspherical_Holder_general} Let $K$ be a compact Lie group with some biinvariant riemannian metric $d$, $U \subset K\times K$ a closed subgroup and $x_0 \in K$. For every character $\chi$ of $U$ that is trivial on the stabilizer of $x_0$ for the action of $U$ by left-right multiplication on $K$, there exists $C_{\chi} \in \R^+$ such that the following holds. For every isometric representation $\pi$ of $K$ on a Banach space $X$ and $\xi,\eta$ unit vectors in $X,X^*$, if the matrix coefficient $c(x) = \langle \pi(x) \xi,\eta\rangle$ satisfies $c(u \cdot x)=\chi(u) c(x)$ for all $x\in K, u \in U$, then 
\[ |c(x) - c(x_0)| \leq C_{\chi}\left( d(x,x_0) +  \| (\int_U \lambda(u \cdot x) - \lambda(u \cdot x_0) du) \otimes \mathrm{Id}_X\|_{B(L^2(K;X))}\right).\]
\end{prop}

The case when $\chi$ is the trivial character is Proposition \ref{prop=all_reduces_to_regular}. For general $\chi$ one uses the following Lemma to reduce to the case $\chi=1$. Here $A(K)$ is the \emph{Fourier algebra of $K$}, that is the space of matrix coefficients of the representation $\lambda$ of $K$ on $L^2(K)$ by left translation, with norm defined by $\|\psi\|_{A(K)} \leq 1$ if and only if there are unit vectors $\xi,\eta \in L^2(K)$ such that $\psi(x) = \langle \lambda(x) \xi,\eta\rangle$.
\begin{lemma}\label{lemma=existence_of_elementinAK_of_type_chi}
Let $\chi$ be a character of $U$ that is trivial on the stabilizer of $x_0$. There exists $\psi \in A(K)$ such that
\begin{itemize}
\item $\psi(u\cdot x) = \overline{\chi(u)} \psi(x)$ for all $x \in K$, $u \in U$.
\item $\psi$ is a Lipschitz function.
\item $\psi(x_0)=1$.
\end{itemize}
\end{lemma}              
Remark that the assumption on the stabilizer of $x_0$ is necessary in this Lemma, because the first implication implies that $\psi(x_0)(1-\overline \chi(u))=0$ for all $u$ in the stabilizer of $x_0$. It is therefore not possible to ensure that $\psi(x_0)\neq 0$ if $\chi$ is not trivial on the stabilizer of $x_0$.
\begin{proof}
If $\psi^0$ belongs to the space of coefficients of a finite dimensional subrepresentation of the representation $\lambda$ of $K$ on $L^2(K)$, then the function
\[\psi(x) = \int_{U} \chi(u) \psi^0(u\cdot x) du\]
also belongs to the space of coefficients of the same finite dimensional representation of $K$, and in particular is $C^\infty$ hence Lipschitz. Moreover it satisfies $\psi(u\cdot x) = \overline{\chi(u)} \psi(x)$. We only have to show that there is a choice of $\psi^0$ such that $\psi(x_0) \neq 0$, since then $\psi/\psi(x_0)$ will work. Assume by contradiction that $\psi(x_0)=0$ for all such $\psi^0$. By the Peter-Weyl theorem the coefficients of finite-dimensional subrepresentations of $\lambda$ form a dense subspace of $C(K)$, and hence $\psi(x_0) = 0$ for all $\psi^0\in C(K)$. But by our assumption on the stabilizer of $x_0$, $\psi^0( u\cdot x_0) = \overline{\chi(u)}$ defines a continuous function on the closed subset $U\cdot x_0$, that we can extend to a continuous map on $K$ and get a contradiction.
\end{proof}

The second Lemma we need expresses that the product of a function in $A(K)$ and of a coefficient of a representation of $K$ on $X$ is a coefficient of $\lambda \otimes \mathrm{Id}_X$ on $L^2(K;X)$.
\begin{lemma}\label{lemma=product_of_AK_and_Bpi}
Let $\pi$ be an isometric representation of $K$ on $X$, $\psi \in A(K)$ and $\xi \in X,\eta \in X^*$. There is $\widetilde \xi \in L^2(K;X), \widetilde \eta \in L^2(K;X^*)$ such that \[\psi(x)\langle \pi(x) \xi, \eta \rangle=\langle (\lambda(x) \otimes \mathrm{Id}_X) \widetilde \xi,\widetilde \eta \rangle\] and $\|\widetilde \xi\| \|\widetilde \eta\| \leq \|\psi\|_{A(K)} \|\xi\|\|\eta\|$.
\end{lemma}
\begin{proof}
Write $\psi(x)=\langle\lambda(x) f_1,f_2\rangle$ for $f_1,f_2 \in L^2(K)$ with $\|f_1\| \|f_2\|=\|\psi\|_{A(K)}$. Then $\psi(x)\langle \pi(x) \xi, \eta \rangle=\langle (\lambda(x) \otimes \pi(x)) (f_1 \otimes \xi),(f_2 \otimes \eta)\rangle$. The operator $V$ on $L^2(K;X)$ defined by $Vf(k)=\pi(k)f(k)$ is an invertible isometry. Moreover $\lambda \otimes \pi = V \circ (\lambda \otimes \mathrm{Id}) \circ V^{-1}$, so that $\psi(x)\langle \pi(x) \xi, \eta \rangle=\langle (\lambda(x) \otimes \mathrm{Id}_X) \widetilde \xi,\widetilde \eta \rangle$ for $\widetilde \xi = V^{-1} (f_1 \otimes \xi)$ and $\widetilde \eta = V^*(f_2 \otimes \eta)$.
\end{proof}

\begin{proof}[Proof of Proposition \ref{prop=nonspherical_Holder_general}] Let $\psi \in A(K)$ given by Lemma \ref{lemma=existence_of_elementinAK_of_type_chi}. Then $(c\psi)(u \cdot x)=(c\psi)(x)$ for all $x \in K, u \in U$. Let $\widetilde \xi \in L^2(K;X)$ and $\widetilde \eta \in L^2(X;K)$ given by Lemma \ref{lemma=product_of_AK_and_Bpi}. For $x \in K$ consider the measure $m_x$ on $K$
\[ m_x(f) = \int_U f(u \cdot x) du\ \ \forall f \in C(K).\]
Then we have $(c\psi)(x) = \langle \lambda(m_x)_X \widetilde \xi,\widetilde \eta\rangle$. The inequality \[\|\widetilde \xi\|_{L^2(K;X)} \|\widetilde \eta\|_{L^2(K;X^*)} \leq \|\psi\|_{A(K)}\] therefore implies that
\[ |\psi(x) c(x)  - \psi(x_0) c(x_0)| \leq \| (\lambda(m_x) - \lambda(m_{x_0}))_X\| \|\psi\|_{A(K)}.\]
But $\psi(x_0) =1$ and $|\psi(x) - 1| \leq \textrm{Lip}(\psi) d(x,x_0)$, so that
\[ |c(x) -c(x_0)| \leq \| (\lambda(m_x) - \lambda(m_{x_0}))_X\| \|\psi\|_{A(K)} + \textrm{Lip}(\psi) d(x,x_0)|c(x)|.\]
This proves the Proposition because $|c(x)| \leq 1$.
\end{proof}

\subsection{Strong Banach property (T)}\label{subsection=strongT}
We now recall Lafforgue's original definition of Strong Banach property (T) and explain why it is equivalent to Definition \ref{def=strongT}. Given a function $M\colon G \to [1,\infty)$ that is bounded on the compacts, and given a class of Banach spaces $\mathcal E$, we can consider the Banach algebra $C_{M,\mathcal E}(G)$, which is the completion of the convolution algebra $C_c(G)$ for the norm $\|f\|_{M,\mathcal E} = \sup_{\pi} \|\pi(f)\|_{B(X)}$ where the supremum is over all representations $\pi$ of $G$ on a space $X \in \mathcal E$ satisfying $\|\pi(g)\| \leq M(g)$. By construction for every such representation there is a natural map still denoted by $\pi\colon C_{M,\mathcal E}(G) \to B(X)$. If the class $\mathcal E$ is stable under complex conjugation then $f \mapsto \overline f$ extends to an isometry of $C_{M,\mathcal E}(G)$, and if $\mathcal E$ is stable under duality and closed subspaces, then the usual involution $f \mapsto \check f$ (with $\check f(g) = \Delta(g^{-1}) f(g^{-1})$ where $\Delta$ is the modular function on $G$) also extends to an isometry of $C_{M,\mathcal E}(G)$, because the contragredient representation of $\pi$ defined in Lemma \ref{lemma=contragredient} satisfies $\|\tr \pi(f)\| = \|\pi(\check f)^*\| = \|\pi(\check f)\|$ for all $f \in C_c(G)$.

Following \cite{MR2574023} let us say that a class of Banach spaces $\mathcal E$ is of type $>1$ if there exists $p>1$ such that $\sup_{X \in \mathcal E} T_p(X)<\infty$.

\begin{dfn}\label{def=strongT_idempotent} A locally compact group $G$ has Strong Banach property (T) if for every length function $\ell$ and every class $\mathcal E$ of type $>1$ stable by duality, closed subspaces and complex conjugation, there exists $s>0$ such that for all $C \in \R^+$, $C_{e^{s \ell + C},\mathcal E}(G)$ contains a real and selfadjoint projection $p$ such that $\pi(p)$ is a projection on $X_\pi^G$ for all representations $\pi$ on $X_\pi \in \mathcal E$ satisfying $\|\pi(p)\| \leq e^{s\ell(g)+C}$.
\end{dfn}
\begin{rem} We prefer to add here the assumption that $\mathcal E$ is stable by subspaces, in order to the ensure that $C_{e^{s \ell + C},\mathcal E}(G)$ is an involutive algebra. Vincent Lafforgue explained to me that the reason why the projection has to be self-adjoint is to ensure that it is unique (there might several projections on $X^G$, but at most one such that $P^*$ is a projection on the vectors fixed by the contragredient representation).
\end{rem}

In the remaining of the paragraph let us explain why (1) $G$ has Strong Banach property (T) in the sense of \ref{def=strongT_idempotent} if and only if (2) $G$ has Strong Banach property (T) with respect to all Banach spaces of nontrivial type in the sense of \ref{def=strongT}. In order to ignore all set-theoretical issues in the argument (due to the fact that the Banach spaces do not form a set), let us only consider the Banach spaces that belong to some fixed set of Banach spaces, that is stable under countable $\ell^2$-direct sum (for example we could take for this set all the isomorphism classes of separable Banach spaces). Assume (1). Fix a length function $\ell$ on $G$. For every integer $k$, denote by $\mathcal E(k)$ the set of all Banach spaces such that $T_{1+1/k}(X) \leq k$. The set $\mathcal E(k)$ is of type $>1$. Hence there is $s>0$ such that for every $C>0$, there is a sequence $f_n$ of functions on $G$ such that $\pi(f_n)$ converges to a projection on $X^{\pi(G)}$ for all representation $\pi$ on $X \in \mathcal E(k)$ satisfying $\|\pi(p)\| \leq e^{s\ell(g)+C}$. Using that the limit of $f_n$ in $C_{e^{s \ell + C},\mathcal E}(G)$ is self-adjoint we can replace $f_n$ by $(f_n+\check f_n)/2$ and assume that the measure $f_n(g)dg$ is symmetric. By a diagonal argument the same sequence can be chosen for all $k$ and $C$. This shows (2) since every Banach space with nontrivial type belongs to $\mathcal E(k)$ for some $k$ and since the length function was arbitrary.

Now assume (2) and fix a length function $\ell$. Take $f \in C_c(G)$ with $\int f=1$. Then the measures $f \ast m_n$ are absolutely continuous with respect to the Haar measure and even correspond to $f_n \in C_c(G)$. Let $\mathcal E$ be a class of type $>1$. By considering a countable $\ell^2$-direct sum of spaces in $\mathcal E$ (which is a space with non trivial type), we see that there is $t>0$ such that for all $C$, $\pi(f_n) = \pi(f) \pi(m_n)$ converges to $\pi(f)\lim_n \pi(m_n) = \lim_n \pi(m_n)$, a projection on $X_\pi^{G}$, for all representations $\pi$ on $X_\pi \in \mathcal E$ such that $\|\pi(g)\|\leq e^{t \ell(g)+C}$. Moreover for every $C$, by considering a countable $\ell^2$-direct sum of such spaces $X_\pi$, this convergences holds uniformly over all representations $\pi$ on $X_\pi \in \mathcal E$ such that $\|\pi(g)\|\leq e^{t \ell(g) +C}$. This proves that $f_n$ converges in $C_{e^{s \ell + C},\mathcal E}(G)$ to an idempotent $p$ such that for all $\pi$, $\pi(p)$ is a projection on the invariant vectors. If moreover $\mathcal E$ is stable under duals and subspaces, by considering the contragredient representation and using that the measure $m_n$ is symmetric, we see that $p$ is self-adjoint. Also if $\mathcal E$ is stable under complex conjugation, $p$ is also real. This shows (1).

\subsection{The harmonic analysis of the pair $(\SO(3,\R), \SO(2,\R))$}
Denote by $(K,U)$ the pair $(\SO(3,\R), \SO(2,\R))$ for the inclusion $A
\in \SO(2,\R) \mapsto \begin{pmatrix}1&0\\ 0&A\end{pmatrix}$. $\SO(2)$ acts transitively on the unit circle in $\R^2$. Hence $U \backslash K/U$ can be identified with the segment $[-1,1]$, through the identification of $U x U$ with $x_{1,1}$. We will sometimes also consider the subgroup $\begin{pmatrix}\SO(2)&0\\ 0&1\end{pmatrix}$ of $K$, that we will denote $\widetilde U$.

Let $\lambda$ be the left regular representation of $K$ on $L^2(K)$.
For $x \in K$, the operator $\lambda(UxU)$ depends only on $x_{1,1}$. We therefore define, for $\delta \in [-1,1]$, the operator $T_\delta$ on $L^2(K)$ by 
\begin{equation}\label{eq=def_of_Tdelta} T_\delta = \lambda(UxU) \textrm{ for all }x \in K\textrm{ such that }x_{1,1}=\delta.\end{equation}

Remark that for every Banach space $X$, $\|(T_\delta)_X\| = 1$.

$T_\delta$ preserves $L^2(U\backslash K)$ and is zero on its
orthogonal, so that $T_\delta$ can be viewed as an operator on
$L^2(U\backslash K)$. Under the identification $U\backslash K \simeq
\sphere$ (with $ U k \mapsto k^{-1} e_1$) $T_\delta$ corresponds to
the following operator on $L^2(\sphere)$. For a continuous function
$f$ on $\sphere$, $T_\delta f(x)$ is the mean value of $f$ on the
circle $\{ y \in \sphere, \langle x, y \rangle =
\delta\}$. Probabistically $T_\delta$ is the Markov operator for the
Markov chain on $\sphere$ that jumps from a point $x$ to a point
uniformly distributed on the circle of euclidean radius $\sqrt{2 - 2\delta}$
around $x$.

Lafforgue proved in \cite[Lemme 2.2]{MR2423763} that the operators $T_\delta$ are selfadjoint, commute and their eigenvalues are $P_n(\delta)$ with mutliplicity $2n+1$ for $n \geq 0$. Here $P_n$ is the $n$-th Legendre polynomial normalized by $P_n(1)=1$. He also proved that $|P_n(0) - P_n(\delta)|\leq C |\delta|^{1/2}$ (by \cite[Lemma 3.11]{HaadL}, one can take $C=4$). This implies that $\|T_0-T_\delta\|_{B(L^2(K))} \leq 4 |\delta|^{1/2}$ for every $\delta \in [-1,1]$. Since for every Hilbert space $H$ and every operator $T$ on $L^2(K)$, $\|T_H\| = \|T\|$, Lemma \ref{prop=all_reduces_to_regular} implies that every normalized $U$-biinvariant matrix coefficient $c$ of a unitary representation of $K$ satisfies 
\begin{equation}\label{eq=U-invariant_coeff_hoelder} |c(x_0) - c(x)| \leq 4 |x_{1,1}|^{1/2}. \end{equation}

In \cite{MR2838352} we proved the following strengthening, that will also be crucial here.
\begin{lemma}\label{lemma=T_delta_in_Sp} For $\delta \in (-1,1)$, $T_\delta \in S^p$ if $p>4$. For $4 <p \leq \infty$ there is a constant $C_p \in \R_+$ such that for $\delta \in [-1/2,1/2]$,
\[ \|T_\delta - T_0\|_{S^p} \leq C_p |\delta|^{\frac 1 2 - \frac 2 p}.\]
\end{lemma}


\section{Proof of Theorem \ref{thm=main_result_Banach}}\label{section=main_part}

A common important ingredient in the two proofs I know that some expanders are superexpanders (which means that they do not coarsely embed into a superreflexive Banach space) is a result of the form $\|A \otimes \mathrm{Id}_X\|_{B(L^p(\Omega;X)}<<1$ for some specific operator $A \in B(L^2(\Omega))$ of small norm and a superreflexive (or more generally of type $>1$) Banach space $X$. Both these results rely on some work relating the geometry of the Banach space $X$ to the boundedness of $A \otimes \mathrm{Id}_X$. For Lafforgue's construction through Banach strong property (T) for $SL(3)$ over a non-Archimedean local field, this result was \cite[Lemme 3.3]{MR2574023}, relying on \cite{MR675400}. For the construction by Manor and Mendel, this was \cite[Theorem 5.1]{manormendel}, relying on \cite{MR647811}. 

Theorem \ref{thm=main_result_Banach}, that we prove in this section, is a result of the same kind. It also relies on some work from the $80$'s namely \cite{MR576647}. Unfortunately the results from \cite{MR576647} do not seem to be enough to prove the conclusion of Theorem \ref{thm=main_result_Banach} for every space $X$ of non trivial type. This is related to questions related to type and cotype that had been left open, see the discussion in \S \ref{subsec=onestepfurther}.

\subsection{$\theta$-Hilbertian Banach spaces~: Pisier's Theorem}\label{subsection=thetaHilb}
Following \cite{MR555306} (see also \cite{MR2732331}), if $\theta \in (0,1]$, we say that a Banach space $X$ is (isomorphically) strictly $\theta$-Hilbertian if there is an interpolation pair $(X_0,X_1)$ \cite{MR0482275} such that $X_1$ is  a Hilbert space, and $X$ is isomorphic to the complex interpolation space $[X_0,X_1]_\theta$ (see \S \ref{subsection=interpolation} for reminders on interpolation). If $1<p<\infty$, by the Riesz-Thorin theorem recalled in \S \ref{subsection=interpolation}, $L^p$-spaces and noncommutative $L^p$ spaces are easy examples of strictly $\theta$-Hilbertian spaces (for $\theta>0$ such that $2/(2-\theta) \leq p \leq 2/\theta$)~: depending on the sign of $p-2$, $L^p$ is a complex interpolation space between the Hilbert space $L^2$ and $L^1$ of $L^\infty$. A strictly $\theta$-Hilbertian space is superreflexive. Pisier \cite{MR555306} proved that the converse holds for Banach lattices, and suggested that the converse might hold in general: every superreflexive Banach space is a subspace of a quotient of a strictly $\theta$-Hilbertian space for some $\theta>0$.

In \cite{MR2732331} a slightly more general notion was defined, called $\theta$-Hilbertian spaces. We choose not to work with this generality here. 

As explained in the introduction, a first observation (that was already known to Lafforgue and Pisier) is that $\SL(3,\R)$ has strong property (T) with respect to the class of Banach spaces that are isomorphically $\theta$-Hilbertian for some $\theta>0$ (and more generally with respect to spaces that are isomorphic to subspaces of quotients of ultraproducts of $\theta$-Hilbertian spaces). 

Let us explain this observation, that follows from the proofs of \cite{MR2423763} and of the results of \cite{MR2732331}.

One of the main results of \cite{MR2732331} is that for a Banach space
$X$ and $0< \theta <1$ the following are equivalent (regular operators and the notation $B_r$ were defined in \S \ref{subsection=regular})~:
\begin{enumerate}
\item \label{item=theta-hilbertien} $X$ is $\theta$-Hilbertian.
\item \label{item=theta_hilbertien_dual} There exists $C<\infty $ such
  that for every regular operator $T:L^2(\Omega,\mu) \to
  L^2(\Omega,\mu)$,
\[ \|T_X\| = \| T \otimes id_X\|_{L^2(\Omega; X) \to L^2(\Omega; X)} \leq C \|T\|_{B(L^2)}^\theta
\|T\|_{B_r(L^2)}^{1-\theta}.\]
\end{enumerate}

The implication \ref{item=theta-hilbertien} $\Rightarrow$
\ref{item=theta_hilbertien_dual} is the easiest, and can be summarized (at least for strictly $\theta$-Hilbertian spaces) by the following Lemma. 
\begin{lemma}\label{lemma=stability_under_subquotients_ultrapr} Let $(\Omega,\mu)$ be a measure space, $C\in\R_+$ and $T$ a bounded operator on $L^2(\Omega,\mu)$. The class of Banach spaces $X$ satisfying $\| T_X\| \leq C$ is stable under subspaces, quotients and ultraproducts.

If $(X_0,X_1)$ is a compatible couple and $X_\theta=[X_0,X_1]_\theta$, then
\[ \|T_{X_\theta}\| \leq \|T_{X_0}\|^{1-\theta} \|T_{X_1}\|^{\theta} \leq \|T\|_{B_r(L^2)}^{1-\theta} \|T_{X_1}\|^{\theta}.\]
In particular if $X_1$ is a Hilbert space
\[ \|T_{X_\theta}\| \leq \|T\|_{B_r(L^2)}^{1-\theta} \|T\|_{B(L^2)}^{\theta}.\]
\end{lemma}

The other direction is an impressive result, but for our purposes here
it is a negative result. We will not use it.

As Lafforgue and Pisier noticed, Lemma \ref{lemma=stability_under_subquotients_ultrapr} implies (since $\|T_0 - T_\delta\|_{B_r(L^2)} \leq \|T_0\|_{B_r(L^2)} + \|T_\delta\|_{B_r(L^2)}=2$) that if $X$ is (strictly) $\theta$-Hilbertian, there exists $C \in \R_+$ such that for all $\delta \in [-1,1]$
\[ \|(T_0 - T_\delta)_X\|_{B(L^2(K;X))} \leq C |\delta|^{\theta/2}.\]
\subsection{The class $\mathcal E_r$}
We now give the definition of the class $\mathcal E_r$. Given a Banach space $X$, one can get new Banach spaces with the following operations~: (a) taking a subspace of $X$, (b) taking a quotient of $X$, (c) taking an ultrapower of $X$ and (d) taking $[X_0,X_1]_\theta$ for some compatible couple $(X_0,X_1)$ with $X_1$ isomorphic to $X$, and $\theta \in (0,1)$.
\begin{dfn} For $2<r<\infty$, $\mathcal E_r$ is the smallest class of Banach spaces that is closed under these operations, and that contains all spaces with type $p$ and cotype $q$ satisfying $1/p-1/q < 1/r$.
\end{dfn}

Since Hilbert spaces have type $2$ and cotype $2$, all $\mathcal E_r$ contain the strictly $\theta$-Hilbertian spaces. It follows from the properties of type and cotype that $\mathcal E_r$ is made of spaces with nontrivial type. We see no obstruction for this class to contain all spaces of nontrivial type (this is related to the question asked on the bottom of \cite[page 279]{MR555306}). In \cite{MR907695}, for every $q>2$, Banach spaces were constructed that are not superreflexive but have type $2$ and cotype $q$. Hence each $\mathcal E_r$ contains non superreflexive spaces.

By Remark \ref{rem=real_interpolation}, if we also allowed the real interpolation $[ \cdot,\cdot]_{\theta,p}$ for $\theta \in (0,1]$ and $1<p<\infty$ in the definition of $\mathcal E_r$, we would get a possibly larger class of Banach spaces that still satisfy our main Theorems \ref{thm=T_renforce_Banachique_SL3_vague} and \ref{thm=quantitative_howe_moore_vague}.

\subsection{Bounds on $\|T_X\|$ for $T \in S^p$ and $X$ with good type and cotype} \label{subsec=onestepfurther}
We are now ready to exploit Lemma  \ref{lemma=T_delta_in_Sp}. For this we prove the following Proposition which, together with Lemma \ref{lemma=T_delta_in_Sp} and \ref{lemma=stability_under_subquotients_ultrapr}, clearly implies Theorem \ref{thm=main_result_Banach} when $\delta \in [-1/2,1/2]$. Theorem  \ref{thm=main_result_Banach} is obvious when $|\delta|\geq 1/2$, since $\|(T_0 - T_\delta)_X\| \leq 2$ for every $X$. 

\begin{prop}\label{prop=lien_entre_typecotype_Schatten} Let $p \in [1,2], q \in [2,\infty]$ and $r \in [2,\infty)$ satisfying $1/p-1/q<1/r$. There is a constant $C \in \R_+$ such that the following holds. If $X$ a Banach space with type $p$ and cotype $q$, $(\Omega,\mu)$ is a measure space, $T:L^2(\Omega,\mu) \to L^2(\Omega,\mu)$ a bounded operator such that $T \in S^r$, then
\begin{equation}\label{eq=lien_TX_Schatten} \|T_X\|_{B(L^2(\Omega,\mu;X))} \leq C T_p(X) C_q(X) \|T\|_{S^r}.\end{equation}
\end{prop}

Let us recall some notation from \cite{MR557370}~: for a Banach space $X$ and an integer
$n$, denote
\[e_n(X)= \sup_u \|u_X\|=\| u \otimes id_X\|_{\ell^2(\N; X) \to \ell^2(\N; X)}\]
where the supremum is over all linear maps $u:\ell^2 \to \ell^2$ of norm $1$ and rank at most $n$. By an easy discretization argument, we have $\|u_X\| \leq e_n(X)$ for every measure space $(\Omega,\mu)$ and every linear maps $u:L^2(\Omega,\mu) \to L^2(\Omega,\mu)$ of norm $1$ and rank at most $n$.

The behaviour of $e_n(X)$ as $n \to \infty$ reflects the geometry of the Banach space $X$. The inequality $e_n(X) \leq \sqrt n$ holds for all Banach spaces. On the opposite, $\sup_n e_n(X)<\infty$ if and only of $X$ is isomorphic to a Hilbert space. $X$ has nontrivial type if and only if $e_n(X) = o(\sqrt n)$. However, given $X$ with nontrivial type, it is not known whether $e_n(X) =O(n^\alpha)$ for some $\alpha<1/2$ (see the problem addressed page 9 of \cite{MR557370}). 

For a rank $n$ operator $u$ on a Hilbert space,
$\|u\|_{S^r}\leq n^{1/r} \|u\|$. It is therefore immediate that if $X$ satisfies \eqref{eq=lien_TX_Schatten} for all $T$, then $e_n(X)\lesssim_X
n^{1/r}$. The converse is almost true, as explained to me by Pisier~: if $e_n(X)\leq C n^{1/r-\varepsilon}$ for some $\varepsilon>0$, then \eqref{eq=lien_TX_Schatten} holds (see the proof of Proposition \ref{prop=lien_entre_typecotype_Schatten}).

We will need the following result of K\"onig, Retherford, Tomczak-Jaegermann.
\begin{thm}\cite[Proposition 27.4]{MR993774} \label{thm=KRTJ}
Let $X$ be a Banach space with type $p$ and cotype $q$. Then 
\[ e_n(X) \leq  T_p(X) C_q(Y) n^{1/p-1/q}.\]
\end{thm}

For the proof of Proposition \ref{prop=lien_entre_typecotype_Schatten} we will also use the following 
\begin{lemma}\label{lemma=decomposition}
Let $H$ be a Hilbert space. Every $T \in S^r$ can be decomposed $T = \sum_{k \geq 0} \alpha_k u_k$ (convergence in the norm topology of $B(H)$)
where $u_k$ have rank at most $2^k$ and norm at most $1$, and $\sum_{k
  \geq 0} 2^k |\alpha_k|^r \leq 2 \|T\|_{S^r}^r$.
\end{lemma}
\begin{proof}
Writing $T=U|T|$ the polar decomposition of $T$, we can assume that
$T>0$. We can therefore write $T=\sum_{n \geq 1} \lambda_n p_n$ where
$(p_n)$ is a familiy of rank one orthogonal projections, $(\lambda_n)$
is a nonincreasing sequence in $\ell^p$ and
$\|(\lambda_n)_n\|_{\ell^p} = \|T\|_{S^p}$. Take $\alpha_k =
\lambda_{2^k}$ for all $k \geq 0$ and $\alpha_k
u_k=\sum_{n=2^k}^{2^{k+1}-1} \lambda_n p_n$.
\end{proof}

\begin{proof}[Proof of Proposition \ref{prop=lien_entre_typecotype_Schatten}]
Let $T \in S^r$. Write $T = \sum_k \alpha_k u_k$ as in Lemma
\ref{lemma=decomposition}. If $r'$ is the conjugate exponent of
$r$~: $1/r+1/r'=1$, using Theorem \ref{thm=KRTJ}, H\"older's inequality, and the assumption $1/p-1/q-1/r<0$,
\begin{eqnarray*} \|T_X\| &\leq& \sum_k |\alpha_k| e_{2^k}(X)  \\
& \leq & T_p(X) C_q(X) \sum_{k \geq 0} |\alpha_k| 2^{k(1/p-1/q)}\\
& \leq &  T_p(X) C_q(X) \left(\sum_k 2^k |\alpha_k|^r\right)^{1/r} \left(\sum_k 2^{r'(1/p - 1/q-1/r)k} \right)^{1/r'} \\ 
& \lesssim_{p,q,r} & \|T\|_{S^r}.\qedhere
\end{eqnarray*}
\end{proof}

\section{Lafforgue's approach to rigidity of $\SL(3,\R)$}\label{section=Lafforgue_theorem}

This section presents in details the approach to rigidity of actions of $\SL(3,\R)$ developped in \cite{MR2423763}. It is mainly expository but also contains some new results.

Throughout this section we make the following assumption~:

{\bf Assumption.} $X$ is a Banach space, $C \in \R_+$, $s \in (0,1/2]$ satisfy
\begin{equation}\label{eq=assumption_on_X}
\|(T_0 - T_\delta)_X\| = \| (T_0 - T_\delta) \otimes id_X\|_{L^2(K; X)
  \to L^2(K; X)} \leq C |\delta|^s \ \forall \delta \in [-1,1].\end{equation}
The operators $T_\delta$ on $L^2(K)$ were defined by \eqref{eq=def_of_Tdelta}.

Lafforgue proved (unpublished) that this assumption implies that $\SL(3,\R)$ has strong (T) with respect to $X$. We start by giving a precise statement of this in Theorem \ref{thm=T_renforce_Banachique_SL3_precis}, and also a precise statement of the quantitative Banach-space valued Howe-Moore property in Theorem \ref{thm=quantitative_howe_moore_precis}. Together with Theorem \ref{thm=main_result_Banach} which proves that Banach spaces in the class $\mathcal E_4$ satisfy \eqref{eq=assumption_on_X}, they imply Theorem \ref{thm=T_renforce_Banachique_SL3_vague}  and \ref{thm=quantitative_howe_moore_vague}. We then provide a detailed proof of these results. The proof of Theorem \ref{thm=T_renforce_Banachique_SL3_precis} was not written explicitely by Lafforgue, but it is an adaptation of the proof in the case of Hilbert spaces in \cite[Th\'eor\`eme 2.1]{MR2423763}. We therefore do not claim much novelty. There are however some differences between the original proof and the one presented here, mainly in subsection \ref{subsection=harmonic_analysis} and subsection \ref{subsection=non_Kinvariant_coeffs}. In this last subsection the same proof as in \cite{MR2423763} would have applied for $V_1 \neq V_2^*$ (which is enough for Theorem \ref{thm=T_renforce_Banachique_SL3_precis}), but in order to treat in the same way the case $V_1 = V_2^*$ and get the Howe-Moore property, we give a slightly different proof.

\subsection{Statements}\label{subsection=statements}
We denote $G=\SL(3,\R)$. $K = \SO(3,\R)$ is a maximal compact subgroup. Denote by $\ell$ the length function on $G$ given by \[\ell(g) = \max(\log \|g\|,\log \|g^{-1}\|)\] where the norm is taken as a linear operator on euclidean $\R^3$. This length function is proper, invariant under multiplication on the left or on the right by elements of $K$ and is geodesic in the following sense: for $0 \leq \lambda \leq 1$, every $g \in G$ can be written $g= g_1 g_2$ with $\ell(g_1) = \lambda \ell(g)$ and $\ell(g_2) = (1-\lambda) \ell(g)$. This property implies that every other length function $\ell'$ on $G$ is dominated by $\ell$ in the sense that $\ell' \leq a \ell + b$ for some $a,b \in \R_+$. The precise choice of the length function is not very important; in \cite{MR2423763} Lafforgue works with $\log \|g\| + \log \|g^{-1}\|$, which is more natural with respect to \cite{MR1905394}. Our choice gives more precise bounds and is more natural in view of the proofs below, in which one is led to work on $\{g \in K \diag(e^r,e^s,e^t)K, r\geq s \geq t, \max(r,-t)=\alpha\}$. For our choice of $\ell$, this set coincides with the sphere $\{g, \ell(g)=\alpha\}$. 

For a continuous representation $\pi$ of $G$ on $X$, we will often consider the following condition~:
\begin{equation}\label{eq=small_exponential_growth} 
\exists L \in \R, t<s/2 \textrm{ such that }\| \pi(g)\| \leq L
e^{t \ell(g)}\  \forall g \in G.\end{equation}

We will prove the following two theorems.
\begin{thm} \label{thm=T_renforce_Banachique_SL3_precis} Under Assumption \eqref{eq=assumption_on_X} on the Banach space $X$, let $\pi:G \to B(X)$ be a continuous representation satisfying \eqref{eq=small_exponential_growth}. Then $\pi(KgK)$ converges as $g \to \infty$ in the norm topology of $B(X)$ to a
projection $P$ on the subspace $X^G$ of $\pi(G)$-invariant vectors. Moreover there is a constant $C'$ such that 
\[ \| P - \pi(KgK)\|_{B(X)} \leq C' e^{-(s-2t)\ell(g)} \textrm{ for all $g \in G$.}\]
\end{thm}
Recall that $\pi(KgK)$ was defined as the operator \[\xi \in X \mapsto \iint_{K \times K} \pi(kgk')\xi  dk dk' \in X.\]
Taking for $m_n$ the uniform measure on $K \diag(e^n,1,e^{-n})K$ (this measure is indeed symmetric), this Theorem implies that $G$ has strong property (T) with respect to the Banach spaces that satisfy \eqref{eq=assumption_on_X}.

\begin{thm} \label{thm=quantitative_howe_moore_precis} Let $V_1,V_2$ be irreducible unitary representations of $K$. There is a constant $C_{V_1,V_2}$ such that the following holds. Under Assumption \eqref{eq=assumption_on_X} on the Banach space $X$, let $\pi:G \to B(X)$ be a continuous representation without invariant vectors satisfying \eqref{eq=small_exponential_growth}, and $\xi \in X_{V_1}$, $\eta \in X^*_{V_2}$. Then

\begin{eqnarray}\label{eq=howe_mooretrivial} |\langle \pi(g) \xi,\eta\rangle| \leq \frac{C_{V_1,V_2} C L^4}{s-2t} \|\xi\| \|\eta\| e^{-(s-2t) \ell(g)} & \textrm{if $V_1=V_2 = 1$}
\\ \label{eq=howe_moorenontrivial} |\langle \pi(g) \xi,\eta\rangle| \leq C_{V_1,V_2} C L^4 \|\xi\| \|\eta\| e^{-(s-2t) \ell(g)} & \textrm{otherwise}.
\end{eqnarray}
\end{thm}
\begin{rem}\label{rem=nonoptimal} If $\pi$ is a unitary representation of $G$ on a Hilbert space, then $s=1/2$ (by the case $p=\infty$ in Lemma \ref{lemma=T_delta_in_Sp}) and $t=0$, so that the right-hand side of \eqref{eq=howe_mooretrivial} and \eqref{eq=howe_moorenontrivial} is $\lesssim e^{-\frac 1 2 \ell(g)}$. This is worse than the optimal decay described in \cite{MR1905394}, which is $\lesssim_\varepsilon e^{-(\frac 1 2 - \varepsilon) (\log \|g\| + \log \|g^{-1}\|)} \lesssim_\varepsilon e^{-(\frac 1 2 - \varepsilon) \frac 3 2 \ell(g)}$ for all $\varepsilon>0$.
\end{rem}

Both results are consequences of the following Proposition. We first deduce Theorem \ref{thm=T_renforce_Banachique_SL3_precis} and \ref{thm=quantitative_howe_moore_precis}, and then prove the Proposition.

\begin{prop}\label{prop=quantitative_howe_moore_partial}
Let $V_1,V_2$ be irreducible unitary representations of $K$. Let $X$ be a Banach space satisfying
\eqref{eq=assumption_on_X}, $\pi$ a representation of $G$ on $X$
satisfying \eqref{eq=small_exponential_growth}, and $\xi\in X_{V_1}$,
$\eta \in X^*_{V_2}$. Then

\begin{itemize}\item If $V_1$ and $V_2$ are both the trivial representation, the coefficient $\langle \pi(g) \xi,\eta\rangle$ has a limit $l$ as $\ell(g) \to \infty$ and 
\begin{equation}\label{eq=howe_moore_equivariant} |\langle \pi(g) \xi,\eta\rangle - l| \lesssim  \frac{C L^4}{s-2t} \|\xi\| \|\eta\| e^{-(s-2t) \ell(g)}.\end{equation}

\item If $V_1$ or $V_2$ is nontrivial,
\begin{equation}\label{eq=howe_moore_spherical} |\langle \pi(g) \xi,\eta \rangle| \lesssim_{V_1,V_2} C L^{4} \|\xi\| \|\eta\| e^{-(s-2t)\ell(g)}.\end{equation}
\end{itemize}
\end{prop}
\begin{rem}\label{rem=reduction_to_K_isometric} 
If one is slightly more careful one can replace the terms $L^4$ by $L^2$ (at least) in \eqref{eq=howe_moore_equivariant} (see the proofs in \cite{MR2423763}). But this is probably not very important, since we did not try to estimate the constant $C_{V_1,V_2}$. In the proofs we will restrict ourselves to the case when the restriction to $K$ of $\pi$ is isometric, and we will get a term $L^2$. We can always reduce to this by renorming $X$ with the norm $\|x\|' = (\int_K \|\pi(k) x \|^2 dk)^{1/2}$. This norm satisfies $L^{-1} \|x\|' \leq \|x\| \leq L\|x\|'$. Moreover by definition $(X,\|\cdot\|')$ is isometrically a subspace of $L^2(K;X)$, so that by Fubini $\|(T_0-T_\delta)_X\|_{B(X,\|\cdot\|')} \leq \|(T_0-T_\delta)_X\|_{B(X,\|\cdot\|)}$. This gives an additional factor $L^2$.
\end{rem}

Let us deduce Theorem \ref{thm=T_renforce_Banachique_SL3_precis}.
\begin{proof}[Proof of Theorem \ref{thm=T_renforce_Banachique_SL3_precis}] 
By \eqref{eq=howe_moore_equivariant}, $\pi(KgK)$ has a limit in the norm topology of $B(X)$ as $\ell(g) \to \infty$. Denote by $P$ the limit. From \eqref{eq=howe_moore_equivariant} we deduce  
\[\|P - \pi(KgK)\| = \sup_{\|\xi\|_X \leq 1, \|\eta\|_X^* \leq 1} |\langle \pi(g) \xi,\eta\rangle - \lim_h \langle \pi(h) \xi,\eta\rangle| \lesssim \frac{C L^4}{s-2t} e^{-s-2t) \ell(g)},\] which is the announced inequality. A first remark is that
\begin{equation}\label{eq=pi(KgP)=P}
\pi(Kg)P = P \ \ \ \forall g \in G.
\end{equation} Indeed, $\pi(Kg)P = \lim_{h \to \infty} \pi(K g) \pi(KhK) =\lim_{h \to \infty} \int_K \pi(K gkh K) dk = P$.
The last equality is because $\ell(gkh) \to \infty$ uniformly in $k \in K$ as $h \to \infty$. Making $g \to \infty$, we get that $P^2= P$, \emph{i.e.} $P$ is a projection. It is clear that the image of $P$ contains $X^G$. We now prove the reverse inclusion. Let $\xi \in P(X)$ and $g \in G$. We prove that $\pi(g) \xi = \xi$. As explained in \cite{MR2423763}, this is easy when the representation is isometric and $X$ is stricly convex~: $\xi$ is the average over $K$ of the vectors of same norm $\pi(kg) \xi$. By strict convexity and continuity all these vectors are equal to $\xi$. In the general case one uses \eqref{eq=howe_moore_spherical} ($\xi$ is $K$-invariant), which implies that for every nontrivial unitary representation $V$ of $K$ and every $\eta \in X^*_V$, 
\[ \langle \pi(g) \xi,\eta\rangle= \lim_{h \to \infty} \int_K \langle \pi(gkh) \xi , \eta\rangle dk= 0.\]
If $\eta \in X^*$ is $K$-invariant, $\langle \pi(g) \xi,\eta\rangle = \langle \xi,\eta\rangle$. By linearity and Theorem \ref{thm=peterWeyl}, we therefore have $\langle \pi(g) \xi,\eta\rangle = \langle \xi,\eta\rangle$ for all $\eta \in X^*_{finite}$. This proves $\pi(g) \xi=\xi$ because $X^*_{finite}$ is weak-* dense in $X^*$ by the remark \ref{rem=finite_type_on_dual}. This concludes the proof.
\end{proof}

Let us deduce Theorem \ref{thm=quantitative_howe_moore_precis}.
\begin{proof}[Proof of Theorem \ref{thm=quantitative_howe_moore_precis}] We only have to prove that if $\pi$ has no invariant vector, the $l$ appearing in \eqref{eq=howe_moore_equivariant} is zero. But from the proof of Theorem \ref{thm=T_renforce_Banachique_SL3_precis}, $l = \langle P \xi,\eta\rangle$ where $P$ is a projection on the $G$-invariant vectors, \emph{i.e.} $P=0$.\end{proof}

\begin{rem} The proof of Theorem \ref{thm=quantitative_howe_moore_precis} has the following feature~: if $V_1$ and $V_2$ are not both trivial, \eqref{eq=howe_moorenontrivial} is proved directly and independantly from the other irreducible representations. On the contrary, in the case $V_1=V_2=1$, \eqref{eq=howe_mooretrivial} is a consequence of \eqref{eq=howe_moorenontrivial} for all the other representations.
\end{rem}

The proofs of \eqref{eq=howe_moore_equivariant} and \eqref{eq=howe_moore_spherical} are different, but they can both be decomposed in two parts~: one about the harmonic analysis on the pair $(K,U)$, and one about the combinatorics of the various embeddings of the pair $(K,U)$ in the pair $(G,K)$. The harmonic analysis part is much simpler for \eqref{eq=howe_moore_equivariant}.

\subsection{Harmonic analysis on $(K,U)$}\label{subsection=harmonic_analysis}
Here we derive some consequences of assumption \eqref{eq=assumption_on_X}. Recall that $K=\SO(3,\R)$, that we equip with some biinvariant metric. We will consider the subgroups $U$ and $\widetilde U$ of $K$ defined by
\begin{equation}\label{eq=def__of_U_Utilde} U= \begin{pmatrix} 1 & 0\\0&\SO(2)\end{pmatrix},\  \widetilde U = \begin{pmatrix} \SO(2) & 0\\0&1\end{pmatrix}.
\end{equation} Recall that $U \backslash K/U$ can be identified with the segment $[-1,1]$, through the identification of $U x U$ with $x_{1,1}$. Fix the following family $(x_\delta)_{-1 \leq \delta \leq 1}$ of representatives in each double class
\begin{equation}\label{eq=def_of_x_delta}
x_\delta = \begin{pmatrix}
\delta & -\sqrt{1-\delta^2}&0 \\ 
\sqrt{1-\delta^2} & \delta & 0 \\ 
0&0&1
\end{pmatrix} \in K.
\end{equation}

Recalling \eqref{eq=def_of_Tdelta}, an immediate consequence of Proposition \ref{prop=all_reduces_to_regular} is that if $\pi$ is an isometric representation of $K$ on a Banach space $X$ satisfying \eqref{eq=assumption_on_X}, then
\begin{equation}
\label{eq=Holder_continuity_of_norms}
\| \pi(U x U) - \pi(Ux_0 U)\|_{B(X)} \leq C |x_{1,1}|^s \ \ \forall x \in K.
\end{equation}
This is all we will need from this subsection for the proof of \eqref{eq=howe_moore_equivariant}. For the proof of \eqref{eq=howe_moore_spherical} one needs also some control of non $U$-invariant coefficients, namely
\begin{prop}\label{prop=nonspherical_Holder_forKU} For every pair of characters $\chi_1,\chi_2$ of $U$, there exists $C_{\chi_1,\chi_2} \in \R^+$ such that the following holds. Let $\pi$ an isometric representation of $K$ on a Banach space satisfying \eqref{eq=assumption_on_X} and $\xi,\eta$ unit vectors in $X,X^*$. If the coefficient $c(x) = \langle \pi(x) \xi,\eta\rangle$ satisfies $c(uxu')=\chi_2(u)\chi_1(u') c(x)$ for all $x\in K, u,u' \in U$, then 
\[ |c(x) - c(x_0)| \leq C_{\chi_1,\chi_2} C d(x, x_0)^s .\]
\end{prop}
\begin{proof} The case when $\chi_1=\chi_2$ are the trivial character is \eqref{eq=Holder_continuity_of_norms}. For general $\chi_1,\chi_2$ we apply Proposition \ref{prop=nonspherical_Holder_general} with the subgroup $U \times U$ of $K \times K$. This is legitimate because $(u,u') \in U \times U  \mapsto u x_0 u'^{-1}$ is injective. Recalling \eqref{eq=def_of_Tdelta} and \eqref{eq=assumption_on_X}, Proposition \ref{prop=nonspherical_Holder_general} therefore gives a constant $C_{\chi_1,\chi_2}$ such that
\begin{eqnarray*} |c(x) - c(x_0)|& \leq& C_{\chi_1,\chi_2} \left( \| (\lambda(U x U) - \lambda(Ux_0 U))_X\| + d(x,x_0) \right)\\ &\leq& C_{\chi_1,\chi_2} \left(C d(x, x_0)^s + d(x,x_0)\right).\end{eqnarray*}
This proves the Proposition because $s\leq 1$.
\end{proof}

\subsection{Combinatorics of the embeddings of $(K,U)$ into $(G,K)$} \label{subsection=combinatorics_of_embeddings}
By the $KAK$ decomposition we can identify $K \backslash G/K$ with $\Lambda = \{ (r,s,t) \in \R^3, r\geq s  \geq t, r+s+t=0\}$ through the identification $(r,s,t) \mapsto K \diag(e^r,e^s,e^t) K$. The reverse map sends $K x K$ to $(\log \lambda_1(x),\log \lambda_2(x),\log\lambda_3(x))$ for $(\lambda_i(x))_{i=1,2,3}$ the ordered sequence of singular values. In particular $\lambda_1(x) = \|x\|$ and $\lambda_3(x) = \|x^{-1}\|^{-1}$. In the pictures below we will represent $\Lambda$ in the plane.

For $\alpha \in \R^+$ denote $D_\alpha =\diag(e^{\alpha},e^{-\alpha/2},e^{-\alpha/2})$ and $i_{\alpha}:K \to G$ defined by $i_{\alpha}(x)=D_\alpha x D_\alpha$. Since $D_\alpha$ commutes with every element of $U$, $i_{\alpha}$ induces a map $U \backslash K/U \to K \backslash G/K$. Denote $j_{\alpha}:[-1,1] \to \Lambda$ the corresponding (continuous) map. One can easily check that $j_{\alpha}(-\delta) = j_{\alpha}(\delta)$, so we will only study $j_{\alpha}$ on $[0,1]$. The following Lemma is all we need for the proof of \eqref{eq=howe_moore_equivariant}, and is best described in Figure \ref{picture=image_of_jalpha}.
\begin{lemma}\label{lemma=image_of_j_alpha} The image by $j_\alpha$ of the segment $[0,1]$ is the segment \[[j_\alpha(0)=(\alpha/2,\alpha/2,-\alpha) , j_\alpha(1) = (2\alpha,-\alpha,-\alpha)].\] Let  $-1/2 \leq \varepsilon\leq 1$. If $\delta \geq 0$ satisfies $j_\alpha(\delta) = ((1+\varepsilon) \alpha, -\varepsilon \alpha, - \alpha)$, then $\delta \leq e^{(\varepsilon -1)\alpha}$.
\end{lemma}
\begin{figure}
  \center
  \includegraphics{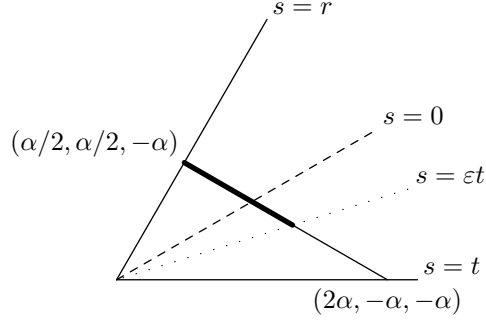}
  \caption{The image of $[0,1]$ by $j_\alpha$. The bold segment is contained in the image of $[0,e^{(\varepsilon-1)\alpha}]$.}\label{picture=image_of_jalpha}
\end{figure}

\begin{proof}
Let $0 \leq \delta \leq 1$. From the equality \[i_\alpha(x_\delta) = \begin{pmatrix}
\delta e^{2\alpha} & -\sqrt{1-\delta^2} e^{\alpha/2}&0 \\ 
\sqrt{1-\delta^2}e^{\alpha/2} & \delta e^{-\alpha} & 0 \\ 
0&0&e^{-\alpha}
\end{pmatrix},\] we see $e^{-\alpha}$ is a singular value of $i_\alpha(x_\delta)$. It is the smallest because $\|i_\alpha(x_\delta)^{-1}\| \leq \|D_\alpha^{-1}\|^2 \|x_\delta^{-1}\| = e^{\alpha}$. This shows that the image of $j_\alpha$ is contained in $\{(r,s,t) \in \Lambda, t=-\alpha\}$, \emph{i.e.} in $[(\alpha/2,\alpha/2,-\alpha),(2\alpha,-\alpha,-\alpha)]$. The equalities $j_\alpha(0)=(\alpha/2,\alpha/2,-\alpha)$ and $j_\alpha(1) = (2\alpha,-\alpha,-\alpha)$ are easy, so that by continuity of $j_\alpha$, we get that the image of $j_\alpha$ is the whole segment $[(\alpha/2,\alpha/2,-\alpha),(2\alpha,-\alpha,-\alpha)]$. Remark also that when $\delta$ grows the angle between the vector $(1,0,0) \in \R^3$ and its image by $x_\delta$ decreases, so that the norm of $i_\alpha(x_\delta)$ increases. This shows that $j_\alpha$ is injective on $[0,1]$.

Let $\delta \geq 0$ such that $j_\alpha(\delta) = ((1+\varepsilon) \alpha, -\varepsilon \alpha, - \alpha)$. Then $\|i_\alpha(x_\delta)\| = e^{(1+\varepsilon) \alpha}$. The norm of a matrix is larger than its $(1,1)$ entry, which gives $e^{2\alpha} \delta \leq e^{(1+\varepsilon) \alpha}$.
\end{proof}

For the proof of \eqref{eq=howe_moore_spherical} (when $V_2 = \overline{V_1}$) we will need to consider maps $i_{\alpha,\beta}:x \in K \mapsto D_\beta x D_\alpha$ with $\alpha \neq \beta$. We also need to control the elements of $K$ appearing in the $KAK$ decomposition of $i_{\alpha,\beta}(x)$. All we need follows from elementary computations with two by two matrices. We summarize them in the following Lemma  (this is almost  contained in Lemme 2.8 in \cite{MR2423763}). Remember the definition $\widetilde U$ in \eqref{eq=def__of_U_Utilde}.

\begin{lemma}\label{lemma=embedding2}
Let $\gamma\geq 0$. Then for every $\alpha \in [\gamma,7\gamma/6]$ and for $i=1,2$ there exist $\delta_i=\delta_i(\gamma,\alpha) \in [0,1]$ and $k_i=k_i(\gamma,\alpha),k_i'=k_i'(\gamma,\alpha) \in \widetilde U$ such that
\begin{enumerate}
\item\label{item=form1} $D_{2\gamma - \alpha} x_{\delta_1} D_\alpha = k_1 \diag(e^{\gamma},1,e^{-\gamma}) k_1'$
\item\label{item=form2} $D_{2\gamma - \alpha} x_{\delta_2} D_\alpha = k_2 \diag(e^{3\gamma/4},e^{\gamma/4},e^{-\gamma}) k_2'$
\item \label{item=contr_delta} $\delta_i \leq e^{-\gamma}$
\item \label{item=contr_k} $\| k_1-1\| \leq 2 e^{-\gamma/4}$, $\| k_1'-1\| \leq 2 e^{-\gamma/4}$, $\| k_2'-1\| \leq 2 e^{-\gamma/4}$
\item \label{item=k2} $k_i(\gamma,\alpha)$ and $k'_i(\gamma,\alpha)$ depend continuously on $\alpha$ and 
\[k_2(\gamma,7\gamma/6) = \begin{pmatrix} 0&-1&0\\1 & 0 &0\\ 0&0&1\end{pmatrix}.\]
\end{enumerate}
\end{lemma}
\begin{proof}
The case $\gamma=0$ is obvious. For any $\alpha \geq \beta >0$, the same argument as for Lemma \ref{lemma=image_of_j_alpha} gives that $\delta \in [0,1] \mapsto K D_\beta x_\delta D_\alpha K$ is injective (for this we need $\beta>0$), and identifies $\{ K D_\beta x_\delta D_\alpha K, 0 \leq \delta \leq 1\} \subset K\backslash G/K$ with the segment \[\left[(\alpha - \frac \beta 2, \beta - \frac \alpha 2 ,- \frac{\alpha + \beta}{2}), ( \alpha + \beta,-\frac{\alpha + \beta}{2}, -\frac{\alpha + \beta}{2}) \right] \subset \Lambda.\]
Note that $(3\gamma/4,\gamma/4,-\gamma)$ belongs to this segment if and only if $\alpha+\beta = 2 \gamma$ and $\gamma \leq \alpha \leq 7\gamma/6$, in which case $(\gamma,0,-\gamma)$ also belongs to it. This proves the existence of $k_i,k_i',\delta_i$ satisfying \ref{item=form1} and \ref{item=form2}. \ref{item=contr_delta} is obtained by saying that the $(1,1)$ entry of the left-hand side of \ref{item=form1} (resp. \ref{item=form2}), namely $e^{2\gamma \delta_i}$, is smaller than the norm of the right-hand side of \ref{item=form1}, namely $e^\gamma$ (resp. \ref{item=form2}, namely $e^{3\gamma/4}$). We now prove \ref{item=contr_k}. Write \[k_i= \begin{pmatrix} \cos \theta_i & -\sin \theta_i&0 \\ \sin\theta_i & \cos \theta_i & 0\\ 0&0&1 \end{pmatrix}\]
and similarly for $k_i'$ with $\theta_i'$. We can assume that $\cos \theta_i' \geq 0$ by replacing $k_i$ and $k_i'$ by $-k_i$ and $-k_i'$. Compute the $\ell^2$-norm of the second column of \ref{item=form2}. to get
\[ (1-\delta_2^2) e^{4\gamma - 3\alpha} + \delta_2^2 e^{-2\gamma} = e^{3 \gamma/2}\sin^2 \theta_2' + e^{\gamma/2} \cos^2 \theta_2'.\]
Using that $\alpha \geq \gamma$, we get $\sin^2 \theta_2' \leq (e^\gamma - e^{\gamma/2})/(e^{3\gamma/2} - e^{\gamma/2}) \leq e^{-\gamma/2}$. It remains to notice that $\| k_2' - 1\| = | e^{i\theta_2'} - 1| \leq  2 |\sin \theta_2'| \leq 2e^{-\gamma/4}$ if $\cos \theta_2' \geq 0$. The same computation gives $\sin^2 \theta_1' \leq e^{-\gamma} \leq e^{-\gamma/2}$, and hence $\| k_1' - 1\| \leq 2e^{-\gamma/4}$. To dominate $\|k_1-1\|$, 
compute the $\ell^2$-norm of the second row in \ref{item=form1} to get
\[ (1-\delta_1^2)e^{3\alpha - 2\gamma} + \delta_1^2 e^{-2\gamma} = \sin^2 \theta_1 e^{2\gamma}  + \cos^2 \theta_1.\]
Using that $\alpha \leq 7\gamma/6$, this gives $\sin^2 \theta_1 \leq (e^{3 \gamma/2 }- 1)(e^{2\gamma} -1) \leq e^{-\gamma/2}$. The inequality $\|k_1 - 1\| \leq 2e^{-\gamma/4}$ follows from the observation that $\cos \theta_1 \geq 0$. This can be checked by comparing $\textrm{Tr}(\diag(e^\gamma,-1,0)A)$ for $A$ both sides of \ref{item=form1}~: $(e^{3\gamma} - e^{-\gamma})\delta_1 = \cos \theta_1 \cos \theta_1' (e^{2\gamma} -1)$, and hence $\cos \theta_1 \cos \theta_1' \geq 0$.

The continuity of $k_i,k_i'$ is simple. The last point \ref{item=k2} is also easy, because for $\alpha = 7\gamma/6$,  
\[ D_{5\gamma/6} x_0 D_{7\gamma/6} = \begin{pmatrix} 0&-1&0\\1 & 0 &0\\ 0&0&1\end{pmatrix} \begin{pmatrix} e^{3\gamma/4}&0&0\\0 & e^{\gamma/4} &0\\ 0&0&e^{-\gamma}\end{pmatrix},\]
which gives $\delta_2(7\gamma/6,\gamma)= 0$ and $k_2(7\gamma/6,\gamma) =\begin{pmatrix} 0&-1&0\\1 & 0 &0\\ 0&0&1\end{pmatrix}$. 
\end{proof}

\subsection{First step: $K$-invariant coefficients}\label{subsection=first_step}
We now prove \eqref{eq=howe_moore_equivariant} in Proposition \ref{prop=quantitative_howe_moore_partial}. As explained in Remark \ref{rem=reduction_to_K_isometric}, we can assume, in addition to the hypotheses of Proposition \ref{prop=quantitative_howe_moore_partial}, that $\pi(k)$ is an isometry of $X$ for every $k \in K$.

We will prove that $g \mapsto \pi(KgK)$ satisfies the Cauchy criterion
with explicit estimates. This will easily imply
\eqref{eq=howe_moore_equivariant}. 

The operator $\pi(KgK) \in B(X)$ only depends on the class of $g$ in $K \backslash G/K$. Denote by $f:\Lambda \to B(X)$ the corresponding map. We will show that $f$ is Cauchy. To do this introduce $d_\pi$ the distance on $\Lambda$ defined by $d_\pi(\lambda,\lambda')=\|f(\lambda) - f(\lambda')\|_{B(X)}$. Take $\alpha>0$, $\varepsilon \in (0,1)$ and denote $\alpha'=(1+\varepsilon)\alpha$, and $\gamma= s -\varepsilon s-2t$. With the notation of Lemma \ref{lemma=image_of_j_alpha}, we have $\pi(K i_\alpha(x) K) = \pi(K D_\alpha) \pi(U x U) \pi(D_\alpha K)$. Hence by \eqref{eq=Holder_continuity_of_norms} and \eqref{eq=small_exponential_growth} and the equality $\ell(D_\alpha)=\alpha$
\begin{eqnarray*}
 d_\pi(j_\alpha(\delta), j_\alpha(0)) &\leq& \|\pi(K D\alpha)\|  C \delta^s\| \pi(D_\alpha K)\|\\
&\leq& C L^2 e^{2\alpha t} \delta^s.
\end{eqnarray*}
In particular if $0 \leq \delta \leq e^{(\varepsilon-1) \alpha}$
\[ d_\pi(j_\alpha(\delta),j_\alpha(0)) \leq C L^2  e^{((\varepsilon - 1)s + 2 t)\alpha)} = C L^2 e^{-\gamma \alpha}.\]
By Lemma \ref{lemma=image_of_j_alpha}, the $d_\pi$-diameter of the segment $\{(r,s,t) \in \Lambda, t=-\alpha, s \geq -\varepsilon \alpha\}$ (Figure \ref{picture=image_of_jalpha}) is less that $2C L^2 e^{-\gamma \alpha}$.

Let us now consider the automorphism $\theta$ of $G$ given by $g \mapsto \tr{(g^{-1})}$. It preserves $K$ and induces  on $\Lambda$ the symmetry around the axis $s=0$, ie the map $(r,s,t) \mapsto (-t,-s,-r)$. Moreover it preserves the length $\ell$. If we apply the preceding to the representation $\pi \circ \theta$, we therefore also have that the $d_\pi$-diameter of the segment $\{(r,s,t) \in \Lambda, r=\alpha, s \leq \varepsilon \alpha\}$ is less that $2C L^2 e^{-\gamma \alpha}$. Combining these two estimates we get in particular that the $d_\pi$-diameter of the domain $\{(r,s,t) \in \Lambda,\alpha \leq \max(r,-t) \leq (1+\varepsilon)\alpha\}$ is less than $6CL^2 e^{-\gamma \alpha}$ (see Figure \ref{picture=zigzag} for a proof). Equivalently, the $d_\pi$-diameter of the domain $\{(r,s,t) \in \Lambda,\alpha \leq \max(r,-t) \leq \alpha'\}$ is less than $6CL^2 e^{(2t-s)\alpha + (\alpha'- \alpha)s}$ for every $\alpha' \in (\alpha,2\alpha)$. By a covering argument this implies, if $\alpha \geq 1$, that the $d_\pi$-diameter of the domain $\{(r,s,t) \in \Lambda, \max(r,-t) \geq \alpha\}$ is less than $6C L^2 \sum_{n \geq 0} e^{(2t-s)(\alpha +n) +s} = C' e^{(2t-s) \alpha}$ where $C' = 6CL^2 e^s/(1-e^{2t-s})$. This diameter goes to zero as $\alpha \to \infty$. This is exactly saying that $f$ is Cauchy. Moreover, if $P \in B(X)$ denotes the limit of $\pi(KgK)$ at infinity, we have $\| \pi(KgK) - P\|_{B(X)} \leq C'e^{(2t-s)\ell(g)}$. This proves \eqref{eq=howe_moore_equivariant}.

\begin{figure}
  \center
  \includegraphics{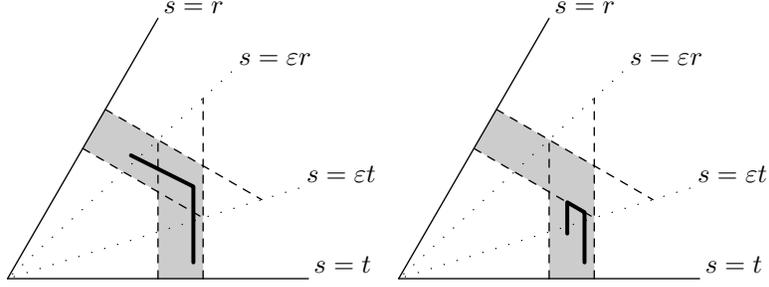}
  \caption{The $d_\pi$-diameter of the region $\{(r,s,t) \in \Lambda,\alpha \leq \max(r,-t) \leq (1+\varepsilon)\alpha\}$ (in grey) is less than $6CL^2 e^{-\gamma \alpha}$~: two points on different sides of the line $s=0$ are linked by two segments of $d_\pi$-diameter less than $2C L^2 e^{-\gamma \alpha}$ (left picture), and two points on the same side are linked by $3$ segments of $d_\pi$-diameter less than $2C L^2 e^{-\gamma \alpha}$ (right picture).}\label{picture=zigzag}
\end{figure}

\subsection{Second step~: non $K$-invariant coefficients.} \label{subsection=non_Kinvariant_coeffs}
We now prove \eqref{eq=howe_moore_spherical}. We only prove the case when $V_1$ is not the trivial representation. The other case is similar. We also assume that the restriction of $\pi$ to $K$ is an isometric representation of $K$ (see Remark \ref{rem=reduction_to_K_isometric}). 

Use Theorem \ref{thm=peterWeyl} to get Banach spaces $Y_1,Y_2$ and isomorphisms $u_1 \colon Y_1 \injtens V_1 \to X_{V_1}$ and $u_2 \colon Y_2 \injtens V_2 \to X_{V_2}^*$ satisfying $\|u_i\|,\|u_i^{-1}\| \leq \mathrm{dim}(V_i)$. By linearity it is enough to prove \eqref{eq=howe_moore_spherical} when $\xi$ and $\eta$ correspond to simple tensors $y_1 \otimes v_1$ and $y_2 \otimes v_2$ for unit vectors $y_1,y_2 \in Y^*$. Let $\iota \colon V_1 \to X$ the $K$-equivariant map $v \in V_1 \mapsto u_1(y_1 \otimes v)$ and $q:X \to V_2^*$ the restriction to $X$ of the adjoint of the map $v \in V_2 \mapsto u_2(y_2 \otimes v) \in X^*$. Then $\|\iota \| \leq \mathrm{dim}(V_1)$ and $\|q \| \leq \mathrm{dim}(V_2)$. We will prove that
\[ \| q \circ \pi(g) \circ \iota\|_{B(V_1,V_2^*)} \lesssim_{V_1,V_2} C L^2 e^{(2t-s)\ell(g)}.\]
This proves \eqref{eq=howe_moore_spherical} (with $L^4$ replaced by $L^2$) because
\[ \langle q \circ \pi(g) \circ \iota (v_1),v_2\rangle = \langle \pi(g) \xi,\eta\rangle\]
and $\|v_1\| \leq \mathrm{dim}(V_i) \|\xi\|$ and  $\|v_2\| \leq \mathrm{dim}(V_2) \|\xi\|$.

The difference with the proof of the first step is that here we fix $\gamma>0$ and we directly prove the preceding inequality for all $g \in G$ such that $\ell(g)=\gamma$ (Figure \ref{picture=image_of_sphere}). 
\begin{figure}
  \center
  \includegraphics{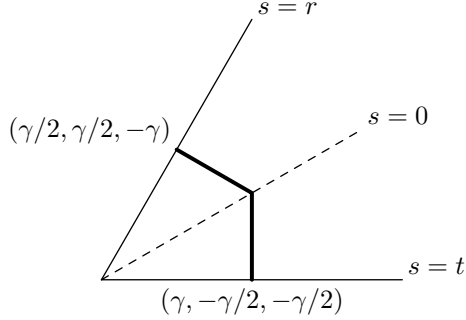}
  \caption{The sphere $\ell(g) = \gamma$. The vertical segment corresponds to $\{g, \ell(g) = \log \|g\| = \gamma\}$. The other segment corresponds to $\{g, \ell(g) = \log \|g^{-1}\| = \gamma\}$.}\label{picture=image_of_sphere}
\end{figure}

Let $\alpha,\beta>0$ and $0 < \delta <1$. Since $U$ is commutative, $V_1$ and $V_2^*$ decompose, as $U$-spaces, as direct sums of characters. Applying Proposition \ref{prop=nonspherical_Holder_forKU} to each of these characters and using the equalities $\ell(D_\alpha)=\alpha$, $\ell(D_\beta)=\beta$, we get
\begin{equation}\label{eq=Hoelder_continuity_general} \| q \circ (\pi(D_\alpha x_\delta D_\beta) - \pi(D_\alpha x_0 D_\beta)) \circ \iota \|_{B(V_1,V_2^*)} \lesssim_{V_1,V_2} C L^2 \delta^s  e^{t(\alpha + \beta)}.\end{equation}

Fix $\gamma\geq 0$ and denote by $A_\gamma = q \circ \pi(\diag(e^\gamma,1,e^{-\gamma})) \circ \iota \in B(V_1,V_2^*)$. Taking $\alpha = \beta = \gamma$ in the preceding equation, Lemma \ref{lemma=image_of_j_alpha} and the fact that $\| q \pi(g) \iota\|$ only depends on $KgK$ imply (see Figure \ref{picture=image_of_sphere}) that
\begin{equation*}\label{eq=suffit_de_traiter_gamma} \| q \pi(g) \iota\| \lesssim_{V_1,V_2} CL^2  e^{\gamma(2t-s)} + \| A_\gamma \|\textrm{ if }g \in K\textrm{ with } \ell(g) = \log \|g^{-1}\| = \gamma.\end{equation*}
This equation applied to the representation $\pi \circ \theta$ where $\theta(g) = \tr{g^{-1}}$ implies that the same inequality holds if $\ell(g) = \log \|g\| = \gamma$. It is therefore enough to prove that $\| A_\gamma \|_{B(V_1,V_2^*)} \lesssim_{V_1,V_2} CL^2  e^{\gamma(2t-s)}$. This will be easy from
\begin{lemma} \label{lemma=A_almost_U_invariant} Let $\chi_1,\chi_2$ be characters of $\widetilde U$, with $\chi_2$ not the trivial character. For a representation $\pi$ of $K$, denote by $\pi(\overline \chi_i) = \int_{\widetilde U} \overline \chi_i(\widetilde u) \pi(\widetilde u) d\widetilde u$. Denote by $\pi_1$ and $\pi_2$ the unitary representations on $V_1$ and $V_2^*$. Then
\[ \| \pi_2(\overline \chi_2) A_\gamma \pi_1(\overline \chi_1)\|_{B(V_1,V_2^*)} \lesssim_{V_1,V_2} CL^2  e^{\gamma(2t-s)}.\]
\end{lemma}
Before proving this Lemma, let us explain how it implies \eqref{eq=howe_moore_spherical}.
\begin{proof}[End of proof of \eqref{eq=howe_moore_spherical}]
As explained above, we have to prove that \begin{equation}\label{eq=norme_Agamma_small}\| A_\gamma \|_{B(V_1,V_2^*)} \lesssim_{V_1,V_2} C L^2  e^{\gamma(2t-s)}.\end{equation}

Let $\chi_{1,2},\dots,\chi_{n,2}$ be the distinct nontrivial characters of $\widetilde U$ that appear in the decomposition of $V_2^*$ as direct sums of characters of $\widetilde U$. Denote similarly $\chi_{1,1},\dots,\chi_{m,1}$ the characters of $\widetilde U$ that appear in the decomposition of $V_1$ as direct sums of characters (we allow here the trivial character). Then $\sum_{k=1}^n \pi_2(\overline \chi_{k,2}) = \mathrm{Id}_{V_2} - \pi_2(\widetilde U)$ and $\sum_{j=1}^m \pi_1(\overline \chi_{j,1}) = \mathrm{Id}_{V_1}$. Summing the inequality from Lemma \ref{lemma=A_almost_U_invariant} therefore gives
\begin{multline}\label{eq=A_almost_U_invariant}\| A_\gamma - \pi_2(\widetilde U) A_\gamma\|_{B(V_1,V_2^*)} = \|\sum_{k=1}^m\sum_{j=1}^n \pi_2(\overline \chi_2) A_\gamma \pi_1(\overline \chi_1)\|_{B(V_1,V_2^*)}\\ \leq \sum_{k=1}^m\sum_{j=1}^n \| \pi_2(\overline \chi_2) A_\gamma \pi_1(\overline \chi_1)\|_{B(V_1,V_2^*)}   \lesssim_{V_1,V_2} CL^2  e^{\gamma(2t-s)}.\end{multline}

Consider the representation $\pi \circ \widetilde \theta$ of $G$ where
\[\widetilde \theta(g) = \begin{pmatrix}0&0&1\\0&1&0\\-1&0&0\end{pmatrix} \tr{g^{-1}} \begin{pmatrix}0&0&1\\0&1&0\\-1&0&0\end{pmatrix}^{-1}.\]
Since $\widetilde \theta$ maps $\widetilde U$ on $U$ and preserves $\diag(e^\gamma,1,e^{-\gamma})$, \eqref{eq=A_almost_U_invariant} applied to this representation gives 
\[\| A_\gamma - \pi_2(U) A_\gamma\|_{B(V_1,V_2^*)} \lesssim_{V_1,V_2} C L^2  e^{\gamma(2t-s)}.\]
If $\xi \in V_1$, these two inequalities become
\[ \|A_\gamma\xi - \pi_2(U) A_\gamma\xi\|_{V_2^*} + \|A_\gamma\xi - \pi_2(\widetilde U) A_\gamma\xi\|_{V_2^*} \lesssim_{V_1,V_2}  C L^2  e^{\gamma(2t-s)} \|\xi\|.\]
Lemma \ref{eq=a_almost_UandU'_invariant} implies \eqref{eq=norme_Agamma_small} and concludes the proof of \eqref{eq=howe_moore_spherical}.
\end{proof}

We can now proceed to the
\begin{proof}[Proof of Lemma \ref{lemma=A_almost_U_invariant}]
We can assume that $\chi_1$ ($\chi_2$) appears in the decomposition of $V_1$ ($V_2^*$) as sum of characters of $\widetilde U$. Indeed otherwise $\pi_1(\overline \chi_1)=0$ on $V_1$ (respectively $\pi_2(\overline \chi_2)=0$ on $V_2^*$). Hence there are only finitely many pairs of characters to consider. For each of them we prove \[\| \pi_2(\overline \chi_2) A_\gamma \pi_1(\overline \chi_1)\|_{B(V_1,V_2^*)} \lesssim_{V_1,V_2,\chi_1,\chi_2} CL^2  e^{\gamma(2t-s)},\] which will prove the Lemma.

For $\alpha \in [\gamma,7\gamma/6]$ take $\delta_i$ and $k_i,k_i'$ given by Lemma \ref{lemma=embedding2} (they depend on $\gamma,\alpha$). Then \eqref{eq=Hoelder_continuity_general} and the inequality $\delta_i \leq e^{-\gamma}$ imply
\[\| \pi_2(k_1) A_\gamma \pi_1(k_1') - \pi_2(k_2) q \pi\begin{pmatrix}e^{\frac{3\gamma}{4}}&0&0\\0&e^{\frac \gamma 4} &0\\0&0&e^{-\gamma}\end{pmatrix} \iota \pi_1(k_2')\| \lesssim_{V_1,V_2} CL^2 e^{\gamma(2t-s)}.\]
Denote $x_\alpha = k_2(\gamma,\alpha)^{-1}k_1(\gamma,\alpha)$ and $x_\alpha' = k_1'(\gamma,\alpha) k_2'(\gamma,\alpha)^{-1}$, so that
\[ \|  \pi_2(x_\gamma) A_\gamma \pi_1(x_\gamma') - \pi_2( x_\alpha) A_\gamma \pi_1(x_\alpha')\| \lesssim_{V_1,V_2} CL^2 e^{\gamma(2t-s)}.\]
If $y_\alpha = x_\gamma^{-1}x_\alpha$ and $y_\alpha'=x_\alpha' {x'_\gamma}^{-1}$ this becomes
\[ \| A_\gamma - \pi_2(y_\alpha) A_\gamma \pi_1(y_\alpha')\|\lesssim_{V_1,V_2} C L^2 e^{\gamma(2t-s)}.\]
Multiply on the left (right) by $\pi_2(\overline \chi_2)$ (resp. $\pi_1( \overline \chi_1)$). Using that $\pi_i(\overline \chi_i)\pi_i(y) = \chi_i(y)\pi_i(\overline \chi_i)$ for every $y \in \widetilde U$, we get
\[ |1-\chi_1(y_\alpha')\chi_2(y_\alpha)| \| \pi_2(\overline \chi_2) A_\gamma \pi_1(\overline \chi_1)\|_{B(V_1,V_2^*)} \lesssim_{V_1,V_2} C L^2  e^{\gamma(2t-s)}. \]
It remains to notice that for $\gamma$ large enough there exists $\alpha \in [\gamma,7\gamma/6]$ such that $|1-\chi_1(y_\alpha')\chi_2(y_\alpha)|=1$. This follows by Lemma \ref{lemma=embedding2} and the mean value theorem. Indeed, we can write $y_\alpha = \begin{pmatrix} \cos \theta_\alpha & -\sin \theta_\alpha & 0 \\ \sin \theta_\alpha & \cos \theta_\alpha & 0 \\ 0&0&1\end{pmatrix}$ and same for $y_\alpha'$ with $\theta_\alpha$ and $\theta_\alpha'$ depending continuously on $\alpha$. Since $y_\gamma= y'_\gamma=1$, we can also assume that $\theta_\gamma=\theta_\gamma'=0$. By Lemma \ref{lemma=embedding2} $|\theta'_\alpha| \lesssim_{V_1,V_2} e^{-\gamma/4}$, and $|\theta_{7\alpha/6}| \geq \pi/2 +O(e^{-\gamma/4})$. Let $k_1,k_2 \in \Z$ such that for $j=1,2$ the character $\chi_j$ is given by $\chi_j \begin{pmatrix} \cos \theta & -\sin \theta & 0 \\ \sin \theta & \cos \theta & 0 \\ 0&0&1\end{pmatrix}= e^{i k_j \theta}$. Then $\chi_1(y_\alpha')\chi_2(y_\alpha) = e^{i(k_1 \theta'_\alpha+k_2 \theta_\alpha)}$. For $\alpha = \gamma$, $k_1 \theta'_\alpha+k_2 \theta_\alpha=0$ whereas for $\alpha = 7 \gamma/6$, $|k_1 \theta'_\alpha+k_2 \theta_\alpha| \geq |k_2| \pi/2 + O(e^{-\gamma/4})$. $k_2$ being nonzero, for $\gamma$ large enough there is indeed $\alpha \in [\gamma,7\gamma/6]$ such that $|k_1 \theta'_\alpha+k_2 \theta_\alpha| = \pi/3$. This proves the Lemma.
\end{proof}

\begin{lemma}\label{eq=a_almost_UandU'_invariant} Let $\sigma:K \to \mathcal U(V)$ be a non trivial irreducible unitary representation of $K$. There is a constant $C_{V}$ such that for every $a \in V$,
\[ \|a\|_{V} \leq C_{V}(\|a - \sigma(U) a \|_V + \|a - \sigma(\widetilde U)a \|_V).\]
\end{lemma}
\begin{proof}
$\sigma(U)$ (resp. $\sigma(\widetilde U)$) is the orthogonal projection on the space of $U$ (resp. $\widetilde U$)-invariant vectors. Since $U$ and $\widetilde U$ generate $K$ as a subgroup, their intersection is the space of $K$-invariant vectors, ie $\{0\}$. The Lemma therefore follows from the compactness of the unit sphere in $V$.
\end{proof}
\begin{rem}By \cite[Lemma 3.10]{liao}, we can take $C_V \leq 3$.
\end{rem}

\bibliographystyle{plain}
 \bibliography{biblio}

\end{document}